\documentclass{article}
\usepackage[english]{babel}
\usepackage{todonotes}
\usepackage{listings}
\usepackage{color} 
\definecolor{mygreen}{RGB}{28,172,0} 
\definecolor{mylilas}{RGB}{170,55,241}
\usepackage{graphics} 
\usepackage{graphicx}
\usepackage{caption}
\usepackage{subcaption}
\usepackage{amsmath}
\usepackage{amssymb}
\usepackage{mathtools}
\usepackage{overpic}
\usepackage{amsthm}
\usepackage{changepage}
\usepackage{hyperref}
\usepackage{enumerate}%

\usepackage{booktabs,array}

\usepackage{pgfplots}
\pgfplotsset{grid style={dotted,gray}}
\pgfplotsset{compat=newest}
\pgfplotsset{plot coordinates/math parser=false}
\newlength\figureheight
\newlength\figurewidth

\usepackage{stfloats} 
\usepackage{array, multirow}

\usepackage[pass]{geometry}
\usepackage[outdir=./]{epstopdf}
\usepackage{url}
\usepackage{tabstackengine}
\usepackage{algorithm} 
\usepackage[noend]{algpseudocode}
\setstacktabbedgap{1.5ex}
\setstackgap{L}{1.2\normalbaselineskip}
 
\theoremstyle{definition}
\newtheorem{example}{Example}

\newtheorem{remark}{Remark}

\definecolor{KTHblue}{RGB}{25,105,188}
\definecolor{KTHlblue}{RGB}{22,159,219}
\definecolor{KTHyellow}{RGB}{251,186,0}
\definecolor{KTHred}{RGB}{176,9,48}
\definecolor{KTHlred}{RGB}{231,51,57}
\definecolor{KTHgreen}{RGB}{98,146,46}
\definecolor{KTHlgreen}{RGB}{175,202,11}
\definecolor{KTHpink}{RGB}{219,81,151}

\addto{\captionsenglish}{}
\renewcommand{\vec}[1]{\boldsymbol{#1}}
\newcommand{\vecmc}[1]{\boldsymbol{\mathcal{#1}}}

\addto{\captionsenglish}{}

\newgeometry{left=2.5cm,right =2.5cm,top= 2.5cm,bottom = 2.5cm}

\providecommand{\keywords}[1]{\textbf{Key words: } #1}

\usepackage[english]{nomencl}
\makenomenclature
\newcommand{\thickhline}{%
	\noalign {\ifnum 0=`}\fi \hrule height 1pt
	\futurelet \reserved@a \@xhline
}

\usepackage{graphicx,bm,hyperref,amssymb,amsmath,amsthm}
\usepackage{xcolor}

\newcommand{\bi}{\begin{itemize}}
\newcommand{\ei}{\end{itemize}}
\newcommand{\ben}{\begin{enumerate}}
\newcommand{\een}{\end{enumerate}}
\newcommand{\be}{\begin{equation}}
\newcommand{\ee}{\end{equation}}
\newcommand{\bea}{\begin{eqnarray}} 
\newcommand{\eea}{\end{eqnarray}}
\newcommand{\ba}{\begin{align}} 
\newcommand{\ea}{\end{align}}
\newcommand{\bse}{\begin{subequations}} 
\newcommand{\ese}{\end{subequations}}
\newcommand{\bc}{\begin{center}}
\newcommand{\ec}{\end{center}}
\newcommand{\bfi}{\begin{figure}}
\newcommand{\efi}{\end{figure}}
\newcommand{\ca}[2]{\caption{#1 \label{#2}}}
\newcommand{\ig}[2]{\includegraphics[#1]{#2}}
\newcommand{\bmp}[1]{\begin{minipage}{#1}}
\newcommand{\emp}{\end{minipage}}
\newcommand{\bp}{\begin{proof}}
\newcommand{\ep}{\end{proof}}

\newcommand{\mbf}[1]{{\mathbf #1}}

\newcommand{\R}{\mathbb{R}}

\newcommand{\bigO}{{\mathcal O}}
\newcommand{\qqquad}{\qquad\qquad}
\newcommand{\qqqquad}{\qqquad\qqquad}

\newtheorem{thm}{Theorem}

\newtheorem{pro}[thm]{Proposition}

\newcommand{\x}{\mathbf{x}}
\newcommand{\y}{\mathbf{y}}

\newcommand{\qq}{\mathbf{q}}
\newcommand{\n}{\mathbf{n}}

\newcommand{\bal}{\bm{\alpha}}
\newcommand{\bga}{\bm{\gamma}}
\newcommand{\eps}{\varepsilon}

\newcommand{\E}{\R^3\backslash\overline{\Omega}}    
\newcommand{\pO}{\partial\Omega}
\newcommand{\ok}{^{(k)}}
\newcommand{\okp}{^{(k')}}
\newcommand{\SR}{{\cal S}}           

\title{A Method of Fundamental Solutions for Large-Scale 3D Elastance and Mobility Problems}

\author{Anna Broms$^{*,1)}$, Alex H. Barnett$^{2)}$ and Anna-Karin Tornberg$^{1)}$\\\\ $^{1)}$ 
	Department of Mathematics, KTH Royal Institute of Technology,
	Stockholm, Sweden\\\\	
	$^{2)}$ Center for Computational Mathematics, Flatiron Institute,
	New York, United States\\\\
	$^*$e-mail: annabrom@kth.se, https://www.kth.se/profile/annabrom}
\date{\today}

\begin{document}

\maketitle

\begin{abstract}
The method of fundamental solutions (MFS) is known to be effective for solving 3D Laplace and Stokes Dirichlet boundary value problems in the exterior of a large collection of simple smooth objects. Here we present new scalable MFS formulations for the corresponding elastance and mobility problems. The elastance problem computes the potentials of conductors with given net charges, while the mobility problem%
---crucial to rheology and complex fluid applications---%
computes rigid body velocities given net forces and torques on the particles. The key idea is orthogonal projection of the net charge (or forces and torques) in a rectangular variant of a ``completion flow.''
  The proposal is compatible with one-body preconditioning, resulting in well-conditioned square linear systems amenable to fast multipole accelerated iterative solution, thus a cost linear in the particle number. For large suspensions with moderate lubrication forces, MFS sources on inner proxy-surfaces give accuracy on par with a well-resolved boundary integral formulation. Our several numerical tests include a suspension of 10000 nearby ellipsoids, using $2.6\times 10^7$ total preconditioned degrees of freedom, where
  GMRES converges to five digits of accuracy in under two hours on one workstation.\\\\
\end{abstract}


	\section{Introduction}
Systems of microscale rigid particles immersed in viscous fluids describe a wide range of phenomena in nature and industry. Examples include transport or diffusion processes \cite{souzy15,Driscoll2017,Sprinkle2017,Sprinkle2020},
rheology and nonlinear shear thickening \cite{Wang2021,Ge2022,fossbrady,Wang2016},
phase transitions in liquid crystals \cite{RevolJean,Wang2019,Yan2019}, collective order in biological systems or in materials science \cite{Hakansson2014,Tran}, 
 and assemblies of
functionalized nanoparticles, with applications in imaging and drug delivery \cite{Wang2023,Chaparro2023}.
Numerical modeling of the dynamics of such systems at zero Reynolds number (negligible inertia) requires solving at each time-step the so-called {\em mobility problem} for the unknown rigid body motions of every particle, given their net forces and torques, with the Stokes equations governing the flow in the fluid domain.
Efficient solvers for this boundary value problem (BVP) are thus needed which can scale to large particle numbers.

For an overview of methods for solving the Stokes mobility problem, see Maxey \cite{Maxey2017}, and the PhD theses of Bagge \cite{Bagge2023} and Peláez \cite{pelaez2022}. In the literature on approximate methods, such as the rigid multiblob method and Stokesian dynamics, the mobility problem is solved via saddle-point linear systems, where the given net forces and torques on each particle appear as additional constraints. These square systems can be efficiently preconditioned and solved via acceleration by fast summation techniques \cite{USABIAGA2016,Fiore2019}.  Turning to convergent methods involving exact Green's functions, the most popular is boundary integral equations (BIE), in which the flow is represented as a surface layer potential whose jump relation is exploited to give a Fredholm second-kind system \cite{Corona2017,Corona2018,Yan2020,AfKlinteberg2016,Rachh2016,Wang2021}.
While iterative solution with fast summation is effective \cite{AfKlinteberg2016,tornberg2008fast},
the challenges of discretization of on-surface weakly-singular integral operators, and quadratures to accurately evaluate the potential near to particle surfaces, remain.   
For smooth deformations of a sphere, the self-interaction (``one-particle'') operator may be discretized globally with high order or spectral accuracy using either Galerkin \cite{atkinson82} or Nystr\"om \cite{Gimbutas2013,Corona2017} methods.
For spherical particles, symmetry may be exploited to diagonalize the on-surface operators and the close-evaluation problem using vector spherical harmonics \cite{Corona2018}. In state-of-the-art work by Yan et al.~\cite{Yan2020}, this was applied to systems of up to 80,000 spheres, using $4\times 10^7$ degrees of freedom, distributed across 1792 CPU cores.

Also using exact Green's functions---but rarely applied in the Stokes setting---is the method of fundamental solutions (MFS). Here an interior rather than surface source distribution is used, and boundary conditions are directly applied at surface collocation nodes in the least-squares sense
\cite{Ka89,Alves2004,Fairweather2005,Barnett2007,Alves2009,Karageorghis2019,Antunes2022}.
Moving the source away from the surface renders the quadratures for both self-interaction and close-evaluation trivial, but brings the disadvantage of a
rectangular exponentially ill-conditioned self-interaction operator.
Yet, this need not pose a problem if each particle needs only a moderate enough number of unknowns to be amenable to dense direct 
least-squares solution.
The authors recently demonstrated an efficient such solution strategy for the Stokes resistance problem (the velocity BVP that is the {\em inverse} of the mobility problem) for spheres \cite{Broms2024}. There, following Liu and the 2nd author in the Helmholtz setting \cite{Liu2016}, we proposed a one-particle (i.e., block-diagonal) preconditioner via dense factorization of each one-particle MFS matrix, leaving a well-conditioned global square system (with identity diagonal blocks) involving surface velocity unknowns, which is solved iteratively using a Stokes fast multipole method.
When the geometry is a large collection of simple smooth particles,
this MFS approach appears to be as efficient as BIE methods \cite{Barnett2007,Stein2022}
(and, when augmented by image charges, even more so \cite{Broms2024}),
while being much simpler to implement in terms of quadrature and close-evaluation.
Indeed, for Stokes spheres, Appendix~\ref{acc_spheres} shows that the MFS needs a very similar number of unknowns as a BIE,
while avoiding all of the algebra of vector spherical harmonics \cite{Corona2018}.
When compared against the rigid multiblob method, the MFS is also much more accurate, as Appendix \ref{multiblob} shows.

It is thus 
appealing to try to apply the MFS to Stokes mobility problems with smooth particle shapes.
However, a direct application fails because imposing the given net forces and torques leads to a constrained least squares problem. This can be written as a \emph{rectangular} saddle point system, where, in contrast to the rigid multiblob method and Stokesian dynamics, efficient preconditioning and fast summation techniques are difficult to apply. The main contributions of the present work are then: 1) to present a different MFS formulation that is free from such additional constraints, and 2) to show how it may be combined with one-particle preconditioning and a fast iterative solution to tackle multi-particle mobility problems efficiently.
This combination allows, we believe for the first time, mobility problems with $10^4$ non-spherical particles to be solved to several accurate digits on a single shared-memory node.

Our main idea is inspired by ``completion'' or ``compound'' flows in the Stokes literature; see, for example, Power \& Miranda \cite{PowerMiranda} and Pozrikidis \cite{Pozrikidis1992}, following Mikhlin \cite{Mikhlin}.
Recall that a double-layer BIE representation---being incapable of generating non-zero net force or torque---must be augmented by another type of representation
which supplies the given net force and torque.
The latter is usually a single interior Stokeslet and rotlet (e.g., \cite{Kim1991,AfKlinteberg2016}), but may also be a line source \cite{Bagge2021} 
or a single-layer surface potential \cite{malhotra2023}.
In contrast, in the present MFS case, the proxy Stokeslets {\em can} supply net force and torque.
Thus we emulate a zero-net-force-and-torque source by orthogonal projection, then use the projected-out proxy-source subspace itself as the completion source;
one might call this a ``recompleted'' representation.
It can be viewed as an MFS variant of the recent BIE work of Malhotra and the 2nd author \cite[Sec.~4.2]{malhotra2023}
(which is the adjoint of a popular interior traction mobility formulation \cite{karrilakim89,Rachh2016,Corona2017}),
and it will demand both square and rectangular ``ones matrices'' (low-rank matrix perturbations that can remove nullspaces \cite{onesmatrix}).

The entire formulation is simpler for the (scalar Laplace) elastance problem than for the (vector-valued Stokes) mobility problem; therefore, following Rachh \& Greengard \cite{Rachh2016},
we start our exposition and numerical tests with elastance.
This application is of independent interest in electrostatics
with large numbers of smooth conductors of given charges and unknown potentials. Although the MFS is popular in the engineering literature, we have not found prior work applying the MFS to elastance, so believe that this is also a useful contribution.

   \begin{figure}[h!]
			\centering
			\includegraphics[trim = {37cm 23cm 32cm 18cm},clip,width=1\textwidth]{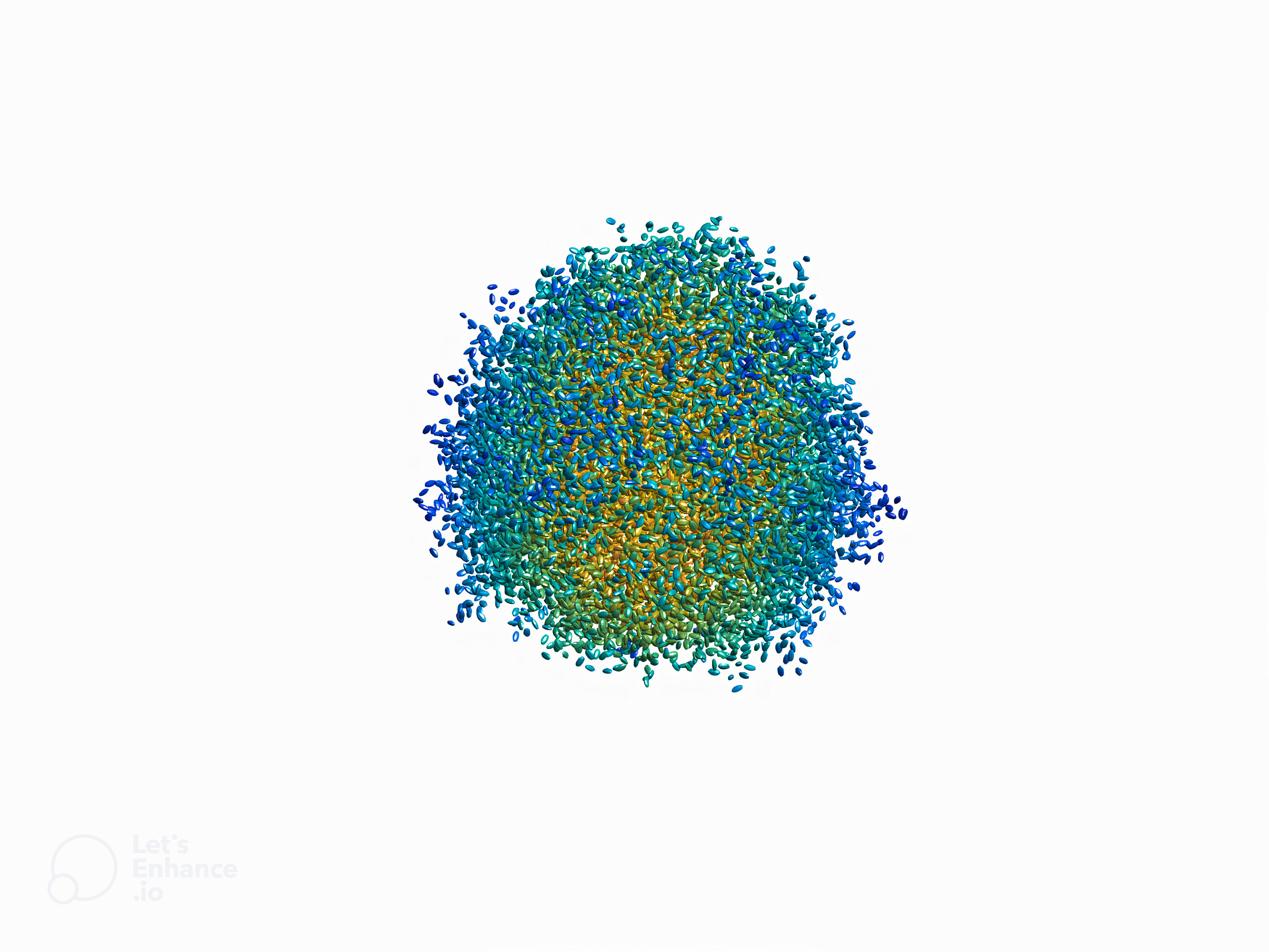}
   \caption{ The mobility problem is solved to 5-digit accuracy for 10000 ellipsoids with semiaxes $\lbrace 0.4,0.6,1 \rbrace$ and minimum separations $\delta = 0.2$. Each particle is discretized with $N=648$ interior Stokeslet sources (1944 unknowns). Color shows surface velocity magnitudes (blue is small and yellow large).}
   \label{large_ex}
   \end{figure}

  \begin{remark}[Close interactions]\label{rem:close}
    Particles moving relative to each other at diminishing separation distances experience increasingly strong lubrication forces. These manifest themselves as peaked force densities near any surface point close-to-touching with another particle \cite[Ch.~9]{Kim1991}, resulting in high resolution requirements, both for BIE methods and for methods based on a volumetric grid \cite{Kim1991,Lefebvre2021}. The same high resolution requirements appear for closely interacting conductors in the Laplace capacitance and elastance problems \cite{Cheng1998}. However, recent MFS work by the authors \cite{Broms2024} tackled this in the Stokes resistance problem for spheres, via discretized line-sources that approximate pair-wise infinite reflection image series, enabling controllable accuracy with few unknowns, down to separations of $10^{-3}R$, where $R$ is the radius.
 For simplicity, we do not attempt to incorporate such images into the presented mobility framework, but expect that it will be straightforward, and leave it for future work.
We thus confine our tests to moderate particle separations ($\ge 0.05 R$), as in other work \cite{Yan2020}.
\end{remark}

As a motivational example, we solve the mobility problem
for the random cluster of 10000 ellipsoids shown in Fig.~\ref{large_ex},
to 5-digit accuracy in the rigid body velocities.
This needs only a single large-memory workstation.%
The solution converges in 1.64 
 hours, in only 7 GMRES iterations, using about $2.6\times 10^7$ unknowns. For comparison, a system of 131 spheroids was solved in a BIE method using 707,400 unknowns in \cite{AfKlinteberg2016}, but the particles were stationary, making it an easier problem due to the absence of lubrication singularities. Note also that, in contrast to \cite{AfKlinteberg2016}, the ellipsoids here need not be spheroids.

\begin{remark}[Code availability and computational setup]
\textsc{Matlab} code for solving the Stokes BVPs is available at \url{https://github.com/annabroms/StokesMFS3D}.
For the Laplace BVPs, our implementation is written in Julia and can be found at \url{https://github.com/ahbarnett/mfs-mobility}. The motivational example shown in Fig.~\ref{large_ex}, along with most other examples in this paper, was run on an 8-core workstation equipped with two Intel\textsuperscript{\textregistered} Xeon\textsuperscript{\textregistered} E5-2637 v3 3.50GHz CPUs (released in 2014) and 256 GB of RAM.
\end{remark}

\subsection{The capacitance and elastance problems}\label{sec:cap}

We complete this introduction with a mathematical statement of the problems to be solved, starting with the scalar case.
Let $\Omega^{(k)}$, $k=1,\dots,P$, be a collection of smooth bounded
disjoint objects in $\R^3$,
each of which has boundary $\pO^{(k)}$,
and let $\Omega \coloneqq \bigcup_{k=1}^P \Omega^{(k)}$ denote the collection,
and $\pO$ denote the union of all boundaries.
Given boundary voltage data $\eta$, the Laplace Dirichlet boundary value problem (BVP) is to find $u$
such that
\bea
\Delta u &=& 0 \qquad \mbox{ in } \E,
\label{pde}
\\
u &=& \eta\qquad\mbox{ on } \pO,
\label{dbc}
\eea
with the decay condition that $u(\vec x) \to 0$ as $\|\vec x\|\to \infty$, uniformly in angle.

The special case where the data in \eqref{dbc} takes the form
\be
u = \phi^{(k)}\qquad\mbox{ on } \pO^{(k)},
\label{vk}
\ee
where $\phi^{(k)}\in\R$ is a given constant voltage on the $k$th boundary,
is called the {\em capacitance problem}.
One seeks to know the resulting net charges
\be
q^{(k)} \coloneqq -\int_{\pO^{(k)}} u_n \,\mathrm dS_{\vec y},
\label{qk}
\ee
where $u_n\coloneqq\partial u/\partial n = \n\cdot \nabla u$,
and $\n$ is the outward unit normal on the boundary.
The BVP solution defines a linear map from the input voltage vector
$\vec\phi \coloneqq\{\phi^{(k)}\}_{k=1}^P$ to the output charge vector $\qq \coloneqq \{q^{(k)}\}_{k=1}^P$,
i.e., a positive semidefinite capacitance matrix $\vec C\in \R^{P\times P}$ acting as
\be
\vec q = \vec C \vec \phi.
\label{qCv}
\ee
The full matrix $\vec C$ could in principle be extracted by solving the BVP in \eqref{pde}--\eqref{vk} $P$ times, with $u\equiv 1$ on $\partial \Omega^{(k)}$ and $u \equiv 0$ on $\partial \Omega \setminus \partial \Omega^{(k)}$, for each of $k=1,\ldots,P$. However, this is rarely a practical proposition.

The {\em elastance problem} is the inverse of the capacitance problem.
Namely, given net charges $q^{(k)}$, $k=1,\dots,P$, we seek the solution $u$ to the PDE \eqref{pde} that is constant on each boundary, as in \eqref{vk}. The voltages $\phi^{(k)}$ are, however, unknown and must be determined as part of the problem.
If the capacitance matrix $\vec C$ were known, from \eqref{qCv} we would also have $\vec \phi = \vec C^{-1} \qq$, where $\vec C^{-1}$ is called the elastance matrix.
The goal is to solve the elastance problem for $\vec \phi$ given $\vec q$, without constructing nor inverting $\vec C$, with a computational cost linear in $P$.

\subsection{The resistance and mobility problems}

  The Stokes analog of capacitance is the resistance problem, where the $P$ objects represent rigid bodies immersed in a Newtonian fluid of constant viscosity $\mu$; see \cite{Rachh2016,Corona2017}.
  	The Stokes Dirichlet BVP is
\begin{equation}\label{stokeseq}
		\begin{aligned}
			-\mu\Delta \vec u + \nabla p &= \vec 0, && \text{ in }\mathbb R^3\backslash\overline{\Omega},\\
			\nabla\cdot \vec u &= 0,&&\text{ in }\mathbb R^3\backslash\overline{\Omega},\\
			\vec u &= \vec g,  &&\text{ on }\partial\Omega, \\
   \vec u(\vec x) &\to \vec 0, &&\|\vec x\|\to\infty,
		\end{aligned}
	\end{equation} 
	with $\vec u(\vec x)\in \mathbb R^3$ the flow velocity at the location $\vec x$ and $p(\vec x)\in \mathbb R$ the pressure. With the particles centered at coordinates $\lbrace\vec c^{(k)}\rbrace_{k=1}^P$ and moving with rigid body velocities and angular velocities $\lbrace \vec v^{(k)},\vec\omega^{(k)}\rbrace_{k=1}^P$, where $\vec v^{(k)},\vec\omega^{(k)} \in \mathbb R^3$, no-slip boundary conditions are applied such that 
 \begin{equation}
     \vec g(\vec x) = \vec v^{(k)}+\vec\omega^{(k)}\times(\vec x-\vec c^{(k)})
     \quad \text{ on }\partial \Omega^{(k)}, \qquad k=1,\ldots,P.
     \label{rbm}
 \end{equation} 
 Once the solution has been determined, the quantities of interest are the total net force and torque on each particle. With the stress tensor $\vec\sigma$ and surface traction $\vec T$ defined as usual by
 \begin{equation}\label{traction}
     \vec \sigma=-p \vec I + \mu \left( \nabla \vec u + \nabla \vec u^T \right),
     \qquad  
     \vec T =\vec\sigma\cdot\vec n,
 \end{equation} we can extract from a solution pair $(\vec u, p)$ the net force and torque on the $k$th particle as
\begin{equation}\label{forces}
\vec f^{(k)} =\int_{\partial\Omega^{(k)}}\vec T(\vec y) \,\mathrm dS_{\vec y},
\qquad
\vec t^{(k)} = \int_{\partial\Omega^{(k)}} (\vec y-\vec c^{(k)})\times \vec T(\vec y) \,\mathrm dS_{\vec y},
 \end{equation}
analogous to the net charge in \eqref{qk}.
The solution to \eqref{stokeseq}--\eqref{rbm} then defines a linear map from the input velocities $\{\vec v^{(k)}\}_{k=1}^P$ and $\{\vec \omega^{(k)}\}_{k=1}^P$, which we stack into a vector $\vec U \in \mathbb R^{6P}$, to the forces and torques, $\{\vec f^{(k)}\}_{k=1}^P$ and $\{\vec t^{(k)}\}_{k=1}^P$, which we stack into a vector $\vec F \in \mathbb R^{6P}$. This linear map is
\begin{equation}\label{resistance}
\vec F = \vec R \vec U,
  \end{equation}
where the resistance matrix $\vec R$ could in principle be recovered by solving  \eqref{stokeseq}--\eqref{rbm} $6P$ times with different right-hand sides, corresponding to $\vec U$ being columns of the identity matrix.
 
Inverting \eqref{resistance} is termed the {\em mobility problem}:
  one seeks a solution to \eqref{stokeseq}--\eqref{rbm} that matches given forces and torques \eqref{forces}, where the rigid body velocities in \eqref{rbm} are now unknown.
  This has a unique solution \cite{Rachh2016}.
Thus the mobility matrix $\vec M\in\mathbb R^{6P\times 6P}$ exists which maps the stacked vector of forces and torques on all particles to their
velocities, i.e.~$\vec M = \vec R^{-1}$, and $\vec U = \vec M \vec F$.
By energy dissipation considerations $\vec R$, and thus $\vec M$, is positive definite \cite{Kim1991}. 
Our goal is then to solve for $\vec U$ given $\vec F$, in only $\bigO(P)$ computational cost (which thus precludes forming $\vec M$).

\begin{remark}[Mobility generalizations]
  Two generalizations of the above arise in applications. i) There may be a given {\em background flow} $\vec u_\infty$ that is a Stokes solution in $\R^3$, such as a shear flow; then one defines $\vec u + \vec u_\infty$ as the physical flow and finds that $\vec u$ solves the above problem, except with an extra term $-\vec u_\infty|_{\partial\Omega^{(k)}}$ on the right-hand side of \eqref{rbm}.
  ii) Active swimmer particles may have a given nonzero slip velocity \cite{liu2024,USABIAGA2016}, which again adds a given term to \eqref{rbm}.
  These generalizations may be simply handled by corresponding extra right-hand side terms in the presented method, thus will not be discussed further.
\end{remark}

\subsection{Outline}
In Section \ref{Dirichlet}, we first  present MFS formulations for the capacitance and resistance problems, using sources on inner proxy-surfaces and boundary conditions imposed by collocation on the physical surfaces, and review their one-body preconditioning technique \cite{Broms2024}.
The main contribution is Section \ref{new_formulations}, namely novel MFS formulations for the elastance and mobility problems, and their one-body preconditioning. 
Section \ref{numerics} contains numerical studies, including computations of the charge density and traction over the particle surfaces. This starts with spheres in Section \ref{spheres}, where the excellent convergence and conditioning of both the elastance and mobility solvers are demonstrated.
The mobility solver is tested for large clusters of ellipsoidal particles in Section \ref{ellipsoids}, where acceleration with a fast multipole method (FMM) enables linear scaling in the particle number.  We conclude in Section \ref{conclusion}. The appendix compares the MFS convergence rates for spheres to those of the state-of-the-art spherical harmonics scheme of \cite{Corona2018,Yan2020}, and compares the MFS errors to those of the popular rigid multiblob method.

\section{The MFS for Dirichlet boundary value problems}\label{Dirichlet}

Here we summarize the MFS technique for capacitance and resistance, which serves to introduce ideas and notations needed later. (For resistance, this is the basic scheme presented in \cite{Broms2024}.) Consider the $k$th particle, and let $\y\ok_j$, $j=1,\dots,N$, be proxy (source) points inside $\Omega^{(k)}$; these will be chosen to lie on a surface a constant separation $\Delta_{\text{sep}}$ from the physical surface $\pO^{(k)}$.
 Thus, for a sphere of radius $R$, the proxy-surface has radius $R_p = R-\Delta_{\text{sep}}$. 
 Let $\x\ok_i$, $i=1,\dots,M$, be collocation points on $\pO^{(k)}$; these are chosen as the nodes of a high-order accurate quadrature scheme for the surface.   The stacked vectors of source and collocation points on particle $k$ are denoted respectively $\vec Y^{(k)}$ and $\vec X^{(k)}$.
 For high collocation accuracy with the MFS one usually sets $M$ slightly larger than $N$ \cite{Barnett2007,Broms2024}.
 For simplicity of notation we take $N$ and $M$ independent of $k$.

\subsection{The capacitance problem}\label{capacitance}
Recall the Laplace fundamental solution,
\be
 G(\x,\y) = G(\x-\y) = \frac{1}{4\pi \|\x-\y\|},
\ee
where $\x\in\R^3$ is a target point and $\y\in\R^3$ a source point.
This obeys $-\Delta G(\cdot,\y) = \delta_\y$ in the distributional
sense. The block of the MFS matrix $\vec S^{(kk')}$
from sources in body $k'$ to targets on body $k$
has entries
\be \label{Skk}
\vec S^{(kk')}_{ij} = G(\x\ok_i,\y\okp_j), \quad i=1,\dots,M, \quad j=1,\dots,N.
\ee

In the case of a single particle ($P=1$), the MFS then solves in the least-squares sense the formally overdetermined $M\times N$ system 
\be
\vec S^{(11)} \bal^{(1)} = \vec\eta^{(1)},
\label{1sys}
\ee
where $\vec\eta^{(1)}\coloneqq\{\eta(\x^{(1)}_i)\}_{i=1}^M$ is the Dirichlet data at the collocation points of particle 1
and $\bal^{(1)}\coloneqq\{\alpha^{(1)}_j\}_{j=1}^N$ are the unknown proxy strengths (coefficients).
For a general $P\ge1$ this becomes
\be
\begin{bmatrix}
\vec S^{(11)} & \vec S^{(12)} & \dots & \vec S^{(1P)} \\
\vec S^{(21)} & \vec S^{(22)} & \dots & \dots \\ 
\vdots &\vdots & \ddots &\vdots \\
\vec S^{(P1)} & \vec S^{(P2)} & \dots & \vec S^{(PP)} 
\end{bmatrix}\begin{bmatrix}
    \vec \alpha^{(1)} \\ \vec \alpha^{(2)} \\ \vdots \\ \vec \alpha^{(P)}
\end{bmatrix} = 
\begin{bmatrix}
    \vec \eta^{(1)} \\ \vec \eta^{(2)} \\ \vdots \\ \vec \eta^{(P)}
\end{bmatrix}.
\label{2sys}
\ee
In the capacitance problem, here one sets $\vec\eta^{(k)} = \phi^{(k)} \vec 1_M $, with $\mbf{1}_M$ the vector with all $M$ entries being $1$.

Having solved \eqref{2sys} in the least-squares sense for the stacked solution vector
$\bal\coloneqq\{\bal\ok\}_{k=1}^P$, the representation of the
solution is
\be
u(\x) = \sum_{k=1}^P \sum_{j=1}^N \alpha\ok_j G(\x,\y\ok_j),
\qquad \x\in\E,
\label{rep}
\ee
which we abbreviate by the notation 
\begin{equation}\label{rep_short}
    u = \sum_{k=1}^P\SR^{(k)}\bal^{(k)},
\end{equation} 
by analogy with layer potentials ($\SR^{(k)}$ is a single-layer proxy source from body $k$).

The net charges $q^{(k)}$ may be extracted either by evaluation of
\eqref{qk} (which requires accurate quadrature weights
for the set of collocation points, and evaluations of $\nabla G$),
or more conveniently via
\be
q^{(k)} = \sum_{j=1}^N \alpha\ok_j, \qquad k=1,\dots,P,
\label{qkgauss}
\ee
the sum of source strengths in the $k$th body, which follows by Gauss' law for $G$.

Each matrix block $\vec S^{(kk')}$ is a discretization of the 1st-kind layer operator from the $k'$th proxy-surface to the $k$th boundary (no Nystr\"om quadrature weights are needed, since they emerge through the linear solve).
Each block, including the diagonal (self-interaction) blocks $\vec S^{(kk)}$,
becomes exponentially ill-conditioned upon convergence (growing $N$), standard behavior for the MFS \cite{Barnett2007}.
Such ill-conditioning is of course inherited by the global system
\eqref{2sys}, but
in Section \ref{solving_res} we show how to precondition this system to make it square and amenable to iterative solution.

\subsection{The resistance problem}\label{resistance_sec}
The tensor-valued Stokes fundamental solution, known as the Stokeslet, is
	\begin{equation}\label{stokeslet}\index{Stokeslet}
		\vec G(\x,\y) = \frac{1}{8\pi\mu \|\vec x-\vec y\|}\left(\vec I_3+\frac{(\vec x-\vec y) (\vec x-\vec y)^T}{ \|\vec x-\vec y\|^2}\right),
	\end{equation}
 with $\vec I_3 \in \mathbb R^{3\times 3}$ the identity matrix. 
 A flow field obeying the Stokes equations, with the associated decay condition at infinity, can be expressed as 
\be
\vec u(\x) = \sum_{k=1}^P \sum_{j=1}^N  \vec G(\x,\y\ok_j) \vec\lambda\ok_j,
\qquad \x\in\E,
\label{repS}
\ee
where the coefficients $\{ \vec\lambda\ok_j \}_{j=1}^N$, $k=1,\dots,P$, are to be determined such that the boundary conditions in \eqref{stokeseq} are satisfied. We abbreviate the representation in \eqref{repS} by
\begin{equation}\label{stokes_sr}
    \vec u = \sum_{k = 1}^P\SR^{(k)}\vec\lambda^{(k)}.
\end{equation}
Letting $\vec b$ contain the stacked Dirichlet velocity data $\vec g$ evaluated at all collocation nodes $\lbrace \vec X^{(k)} \rbrace_{k=1}^{P}$,
the proxy coefficients $\vec\lambda$ solve the overdetermined $3MP\times 3NP$ system 
\begin{equation}\label{ressyst}
    \vec S\vec \lambda = \vec  b .
\end{equation}

We now specialize to rigid body motion data.
Assume first that $P=1$, and let $\vec U^{(1)} = \left[{\vec {v}^{(1)}}^T, {\boldsymbol\omega^{(1)}}^T\right]^T$ be the vector of rigid body velocities for one particle. Further, let $\vec K_M^{(1)}\in \mathbb R^{3M\times 6}$ be the matrix relating these velocities to particle surface velocities at its $M$ collocation points.  Then, a block row of the matrix $\vec K_M^{(1)}$ determines the velocity at the surface point $\vec x_i^{(1)}$ belonging to the particle, so that
	\begin{equation}
		\vec b_i = (\vec K_M^{(1)} \vec U^{(1)})_I = \vec v^{(1)} + \vec\omega^{(1)}\times(\vec x_i^{(1)}-\vec c^{(1)}),
	\end{equation}
	with $i$th index set $I = \lbrace 3(i-1)+k \rbrace_{k=1}^3$. 
In matrix form, $\vec K_M^{(1)}$ may be written as 
\begin{equation}\label{Kdef}
    \vec K_M^{(1)} = \begin{bmatrix}
        \vec I_3 & (\vec x_1^{(1)}-\vec c^{(1)})_\times \\
        \vec I_3 & (\vec x_2^{(1)}-\vec c^{(1)})_\times \\
        \vdots & \vdots \\
        \vec I_3 & (\vec x_M^{(1)}-\vec c^{(1)})_\times      
    \end{bmatrix},
    \qquad\text{using notation } (\vec d)_\times\coloneqq  \begin{bmatrix} 0 & d_3 &-d_2\\ -d_3 & 0 & d_1 \\ d_2 & -d_1 & 0 \end{bmatrix},
\end{equation}
the skew-symmetric matrix performing a cross-product. 
The least squares problem in \eqref{ressyst} for this one particle (alone in the fluid) hence takes the form $\vec S^{(11)} \vec \lambda^{(1)} = \vec K_M^{(1)}\vec U^{(1)}$.
The net forces and torques on the particle can be determined (analogously to \eqref{qkgauss})
from computed proxy coefficients $\vec \lambda^{(1)}$ via 
	\begin{equation}\label{force_eq}
		\vec f^{(1)} =\sum_{i = 1}^N\vec\lambda_i^{(1)},\quad\vec t^{(1)} = \sum_{i = 1}^N (\vec y_i^{(1)}-\vec c^{(1)})\times\vec\lambda_i^{(1)}.
	\end{equation}
 With more compact notation, \eqref{force_eq} can be written as\begin{equation}\label{force_syst}
		\begin{bmatrix} \vec f^{(1)}\\ \vec t^{(1)}\end{bmatrix} = {\vec K_N^{(1)}}^T\vec\lambda^{(1)},
	\end{equation}
	with the rigid body proxy matrix $\vec K_N^{(1)}\in\mathbb R^{3N\times 6}$ defined as in \eqref{Kdef}
 except that we use the subscript change from $M$ to $N$ to indicate that source points $\vec Y^{(1)}$ have replaced collocation nodes $\vec X^{(1)}$.
 
 Generalizing this to the full system of $P\ge 1$ particles, 
 the matrices $\vec K_M$ and $\vec K_N$ both have a block diagonal structure, with
 \begin{equation}\label{Kbig}
     \vec K_M\coloneqq \begin{bmatrix} \vec K_M^{(1)} & \vec 0 & \dots & \vec 0 \\
     \vec 0 &  \vec K_M^{(2)} & \dots & \vec 0 \\
     \vdots & \vdots & \ddots & \vdots \\
     \vec 0& \vec 0 & \dots & \vec K_M^{(P)} \end{bmatrix},\quad      \vec K_N\coloneqq \begin{bmatrix} \vec K_N^{(1)} & \vec 0 & \dots & \vec 0 \\
     \vec 0 &  \vec K_N^{(2)} & \dots & \vec 0 \\
     \vdots & \vdots & \ddots & \vdots \\
     \vec 0& \vec 0 & \dots & \vec K_N^{(P)} \end{bmatrix}.
 \end{equation}
  We may then solve the global system \eqref{ressyst} in the
  least-squares sense, with the right-hand side given by $\vec b =\vec K_M\vec U$, for $\vec U\in\R^{6P}$ the stacked vector of given translational and angular velocity data. 
 After solution, the desired forces and torques in the stacked vector $\vec F$ can then extracted directly from $\vec\lambda$ using 
 \begin{equation}\label{force_def}
     \vec F = \vec K_N^T\vec\lambda,
 \end{equation}
 analogously to \eqref{qkgauss}, which avoids evaluation of any surface traction integrals.
 Finally, should the traction function be needed on surfaces, it may be determined from \eqref{traction}, via \eqref{repS} and using the pressure solution given by 
 \begin{equation}
     p(\vec x) = \frac{1}{8\pi}\sum_{k=1}^P\sum_{j=1}^N\vec \Pi(\vec x,\vec y_j^{(k)})\cdot\vec\lambda_j^{(k)},
 \end{equation}
 using the standard vector-valued pressure fundamental solution $\vec \Pi(\vec x,\vec y) = 2\dfrac{(\vec x-\vec y)}{\|\vec x-\vec y\|^3}$.

\subsection{One-body preconditioning for large-scale Dirichlet problems}
\label{solving_res}
The full least-squares
systems \eqref{2sys} and \eqref{ressyst} are large, dense, and ill-conditioned.
Here, we describe how they can be preconditioned to allow for an accelerated iterative
solution. We start with the capacitance problem \eqref{2sys}.
Let $\vec S^{(kk)} = \vecmc U^{(k)} \vec\Sigma^{(k)}{\vecmc V^{(k)}}^T$ be the singular value decomposition (SVD) of the self-interaction matrix of particle $k$,
with $\vec\Sigma^{(k)}$ the diagonal matrix with entries the singular values $\sigma_1^{(k)}\ge \sigma_2^{(k)}\ge \dots \sigma_N^{(k)}$.
Then, the solution operator to its one-body linear system
$\vec S^{(kk)} \bal^{(k)} = \vec\eta^{(k)}$
may be approximated using the pseudo-inverse,
\begin{equation}\label{pseudo}
\bal^{(k)} \approx {\vec S^{(kk)}}^+ \vec\eta^{(k)} \coloneqq \vecmc V^{(k)} {\vec \Sigma^{(k)}}^+ \bigl({\vecmc U^{(k)}}^T \vec\eta^{(k)}\bigr),
\end{equation}
where $\vec\eta^{(k)}$ is the Dirichlet data vector,
while ${\vec \Sigma^{(kk)}}^+$ is the diagonal matrix with entries $1/\sigma_j^{(k)}$ when $\sigma_j^{(k)} > \sigma_1^{(k)}\eps_{\text{trunc}}$, or zero otherwise.
Typically, one sets $\eps_{\text{trunc}}$ smaller than the desired error, but, for stability, somewhat larger than machine precision.
Note that, due to catastrophic cancellation, ${\vec S^{(kk)}}^+$ cannot be stably formed then applied as a matrix. Rather, the two-step application in the final expression in \eqref{pseudo} is needed for numerical stability \cite{TrefethenBau,Lai2015,Malhotra2015,Stein2022,Parolin2022}, and is implied whenever we
write ${\vec S^{(kk)}}^+$ in what follows.

We {\em right-precondition} with ${\vec S^{(kk)}}^+$ in the following way \cite{Liu2016,Broms2024}: 
let the new unknowns (which represent surface potentials)
be given by the vectors $\bga\ok = \vec S^{(kk)} \bal\ok \in \mathbb R^M$,
for $k=1,\dots,P$.
To find the system that $\bga\ok$ satisfies,
one substitutes
\begin{equation}
    \bal\ok = {\vec S^{(kk)}}^+\bga\ok, \quad k=1,\dots,P,
    \label{reconbal}
\end{equation}
into the original system \eqref{2sys},
which would result in each diagonal block becoming
$\vec S^{(kk)} {\vec S^{(kk)}}^+$.
Yet the tall (rectangular) nature of $\vec S^{(kk)}$ brings about a twist: $\vec S^{(kk)} {\vec S^{(kk)}}^+\in \R^{M\times M}$ must be strictly rank-deficient, since $N<M$,
which would be very far from creating a well-conditioned system!
Thus, a key step to make the formulation well-conditioned is to replace all diagonal blocks $\vec S^{(kk)} {\vec S^{(kk)}}^+$ by $\vec I$, the $M\times M$ identity matrix.
This is expected to maintain solution accuracy because ${\vec S^{(kk)}}^+$ is presumed to give an accurate set of proxy-strengths for the one-body problem, for all $\bga\ok$ in the subspace of smooth surface vectors,
and the use of $\vec I$ merely enforces good conditioning also in the (irrelevant) complementary subspace. 
Summarizing the above, the preconditioned version of \eqref{2sys} is
\be
\begin{bmatrix}
			\vec I & \vec S^{(12)} {\vec S^{(22)}}^+& \dots & \vec S^{(1P)}{\vec S^{(PP)}}^+ \\
			\vec S^{(21)}{\vec S^{(11)}}^+ & \vec I & \hdots & \hdots \\
			\vdots & \vdots & \ddots & \vdots \\
			\vec S^{(P1)}{\vec S^{(11)}}^+ & \hdots & \hdots & \vec I
		\end{bmatrix} 
  \begin{bmatrix}
      \vec \gamma^{(1)} \\ \vec\gamma^{(2)} \\ \vdots \\ \vec\gamma^{(P)}
  \end{bmatrix} =   \begin{bmatrix}
      \vec \eta^{(1)} \\ \vec\eta^{(2)} \\ \vdots \\ \vec\eta^{(P)}
  \end{bmatrix}.
\label{2sysprec}
\ee


We propose to solve \eqref{2sysprec} iteratively (using, for instance, GMRES), to get the stacked surface value vector
$\bga \coloneqq \{\bga\ok\}_{k=1}^P$. This is outlined in Algorithm \ref{alg:cap}.
For large-scale problems ($PM \gg 10^4$) the
square matrix-vector multiply needed in each iteration requires three steps, where step 2 uses the FMM, as
explained in the function \texttt{MATVEC} in Algorithm \ref{alg:cap}. Finally, from the proxy strengths, one can extract the charges via \eqref{qkgauss}, or use \eqref{rep_short} to evaluate the solution anywhere in the domain, including accurately on boundaries.





\begin{algorithm}
 \algrenewcommand\algorithmiccomment[2][\footnotesize]{{#1\hfill\(\triangleright\) #2}}  
 \algnewcommand{\LeftComment}[1]{\Statex \(\triangleright\) #1}
 
\caption{Fast capacitance solve with one-body preconditioning}\label{alg:cap}
\textbf{Global data:} Proxy point sets $\{ \vec X^{(k)} \}_{k=1}^P$, collocation point sets $\{ \vec Y^{(k)} \}_{k=1}^P$
\begin{algorithmic}

\Function{solve}{$\vec \eta$}
    \State \textbf{Input:} Stacked right-hand side surface voltage data vector $\vec \eta$, 
    \State \textbf{Output:} Stacked proxy strength (charge) vector $\bal$.

\LeftComment{Solve for stacked surface values $\bga$ using matrix-vector multiply function defined below:}
    \State $\bga \gets \mathrm{GMRES}(\texttt{MATVEC}, \vec\eta)$
    \LeftComment{Recover proxy strengths for each body via local pseudo-inverse apply:}
    \For{$k = 1$ to $P$}
     
        \State \hskip1em $\bal^{(k)} \gets 
        \vecmc V^{(k)} {\vec \Sigma^{(k)}}^+ ({\vecmc U^{(k)}}^T \vec\gamma^{(k)})$
        \Comment{see eqn.~\eqref{pseudo}}
    \EndFor
    \State \Return $\bal$
\EndFunction

\vspace{1em}

\Function{matvec}{$\vec \bga$}
    \State \textbf{Input:} Stacked surface values vector $\vec \gamma = \{ \bga^{(k)}\}_{k=1}^P$, 
    \State \textbf{Output:} Stacked surface potential vector $\vec u = \{ \vec u^{(k)} \}_{k=1}^P$.
    \LeftComment{Step 1: Recover proxy source strengths for each body via local pseudo-inverse apply:}
    \For{$k = 1$ to $P$}
        \State \hskip1em $\hat \bal^{(k)} \gets 
        \vecmc V^{(k)} {\vec \Sigma^{(k)}}^+ ({\vecmc U^{(k)}}^T \vec\gamma^{(k)})$ 
    \EndFor
    \LeftComment{Step 2: Fast potential evaluation at all $PM$ targets from all $PN$ sources:}
    \State \hskip1em $\vec u \gets \texttt{FMM\_evaluate}(\{ \vec X^{(k)} \},\{ \vec Y^{(k)} \}, \{\hat{\bal}^{(k)} \})$ \Comment{applies bare MFS matrix from \eqref{2sys}}
    \LeftComment{Step 3: Locally correct to convert to identity diagonal blocks in \eqref{2sysprec}:}
    \For{$k = 1$ to $P$}
        \State \hskip1em $\vec u^{(k)} \gets \vec u^{(k)} - \vec S^{(kk)} \hat{\bal}^{(k)}$ \Comment{subtract local self-contribution}
        \State \hskip1em $\vec u^{(k)} \gets \vec u^{(k)} + \vec\gamma^{(k)}$
        \Comment{add back original input}
    \EndFor
    \State \Return $\vec u$
\EndFunction
\end{algorithmic}
\end{algorithm}


The generalization of the above to the Stokes resistance problem
is straightforward. The diagonal blocks now have size $3M \times 3M$, and the matrix-vector multiply needs an FMM which applies the Stokeslet to vector source strengths. Upon convergence, $\vec\gamma^{(k)}\in\mathbb R^{3M}$ are mapped to $\vec\lambda^{(k)}\in\mathbb R^{3N}$ for $k = 1,\dots, P$, analogously to \eqref{reconbal}, and $\vec\lambda$ is then used to compute forces and torques via 
\eqref{force_def}, to evaluate the fluid flow anywhere in the exterior domain $\mathbb R^3\backslash\overline{\Omega}$, or evaluate the surface traction function.
\begin{remark}[Exploiting congruent or resized particles.]\label{congruent}
In the case where all proxy and collocation point sets are translates of those for a single body, the self-interaction matrices ${\vec S^{(kk)}}$ are identical for all $k$ (see \cite{Broms2024} for the sphere case). More generally, for congruent non-spherical particles with different orientations, or for resized copies of a given shape, a single SVD per unique shape suffices to determine the action of the self-interaction pseudo-inverses ${\vec S^{(kk)}}^+$ required for preconditioning.
In the case of spatial rescaling of a particle (and its discretization) by a factor, then ${\vec S^{(kk)}}^+$ is simply multiplied by this factor.
Rotations need a little more bookkeeping:
let a base particle at the origin with a given reference orientation be discretized by source points stacked in $\vec Y^{(0)}\in\mathbb R^{3N}$ and collocation points stacked in $\vec X^{(0)}\in\mathbb R^{3M}$. Then, a particle with a general orientation described by the rotation matrix $\vecmc R^{(k)}\in\mathbb R^{3\times 3}$ is discretized by $\vec X^{(k)} = \vecmc R^{(k)}_N\vec X^{(0)}$ and $\vec Y^{(k)} = \vecmc R^{(k)}_M\vec Y^{(0)}$. Hence, \begin{equation}\label{Krot}
    \vec K_N^{(k)} = \vec R_N^{(k)}\vec K_N^{(0)}\vec R_2^{(k)} \text{ and } \vec K_M^{(k)} = \vec R_N^{(k)}\vec K_M^{(0)}\vec R_2^{(k)}.
\end{equation} Here, $\vecmc R_N^{(k)}\in\mathbb R^{3N\times 3N}$, $\vecmc R_M^{(k)}\in\mathbb R^{3M \times 3M}$ and $\vecmc R_2^{(k)}\in\mathbb R^{6 \times 6}$ are  matrices with the small matrices $\vecmc R^{(k)}$ in their diagonal blocks. Given $\vec B = \vecmc U^{(0)}\vec\Sigma^{(0)} {\vecmc V^{(0)}}^T$, a factorization of the system matrix for the base particle, $\vec\lambda^{(k)} = {\vec S^{(kk)}}^+\vec\gamma^{(k)}$ is then determined as 
		\begin{equation}
			\vec\lambda^{(k)} = \vecmc R_N^{(k)}\vecmc V^{(0)}{\vec\Sigma^{(0)}}^{+}\biggl({\vecmc U^{(0)}}^T\bigl({\vecmc R_M^{(k)}}^T\vec \gamma^{(k)}\bigr)\biggr).
		\end{equation}
	\end{remark}


\section{Formulations for elastance and mobility}\label{new_formulations}

We now turn to the main contribution:
well-conditioned MFS elastance and mobility formulations free from constraints associated with the given net particle quantities.
We start with the simpler elastance case in Section \ref{elastance}, proceed to mobility
in Section \ref{mobility},
and then explain how 
the resulting systems are effectively solved with one-body preconditioning in Section \ref{one_body_revisit}.

\subsection{Elastance formulation}\label{elastance}
We explain first the case $P=1$ for simplicity.
Let $\vec L = \frac{1}{N}\mbf{1}\mbf{1}^T$ be the $N\times N$ matrix
with all entries $1/N$. It is the orthogonal projector
onto the constant vectors in $\R^N$,
while $\vec I-\vec L$ is the orthogonal projector onto the complement space. 
Recalling the representation notation \eqref{rep}-\eqref{rep_short},
we set up a ``completion potential'' $\SR^{(1)}\bal_0^{(1)}$
where $\bal_0^{(1)}\in\R^{N}$ is a known strength vector
designed to impart the desired net charge $q^{(1)}$ to the body.
The simplest such choice
is the constant vector $\vec \alpha_0^{(1)} = \frac{1}{N}q^{(1)}\vec 1_N$.
In terms of an unknown source vector $\bal^{(1)}\in\R^{N}$,
our proposed representation is
\be
u = \SR^{(1)} (\vec I-\vec L)\bal^{(1)} + \SR^{(1)}\bal_0^{(1)} \qqqquad \mbox{($P=1$ case)}.
\label{erep1}
\ee
We have here projected out the non-zero mean part of $\bal^{(1)}$ so that the
first term cannot change the net body charge (a property reminiscent of a double-layer potential), so that the net charge is $q^{(1)}$ by construction.
Inserting \eqref{erep1} into the boundary condition \eqref{vk}
and enforcing this at all collocation nodes gives
\be
\vec S^{(11)}(\vec I-\vec L)\bal^{(1)} + \vec S^{(11)}\bal_0^{(1)} = \vec\eta^{(1)}
,
\label{esys1}
\ee
where $\vec S^{(11)}$ is the MFS matrix, and again $
\vec\eta^{(1)}= \phi^{(1)}\mbf{1}_M$ some constant
vector, as in \eqref{1sys}, but now with unknown constant.
The system is closed by choosing a representation for $\phi^{(1)}$ 
in terms of
$\bal^{(1)}$. The subspace Span$\{\mbf{1}_N\}$ is available for this, it
having no effect on the first term $\vec S^{(11)}(\vec I-\vec L)\bal^{(1)}$. 
Hence, we let $\vec L_r\in\mathbb R^{M\times N}$ be the rectangular matrix with all entries $1/N$, and  make the ansatz
\be
\vec\eta^{(1)}= -\vec L_r\bal^{(1)}.
\label{ansatz}
\ee
This is inspired by ideas in Stokes mobility (see, e.g., \cite[Sec.~4.2]{malhotra2023}).
Substituting this into \eqref{esys1} gives the linear system
\[
\bigl[ \vec S^{(11)}(\vec I-\vec L) + \vec L_r \bigr] \bal^{(1)} = -\vec S^{(11)}\bal_0^{(1)}.
\]
Note that we do not solve for the constant voltage $\phi^{(1)}$ directly; however, it can easily be extracted from any row of \eqref{ansatz}, to give $\phi^{(1)} = -\frac{1}{N}\mbf{1}^T\bal^{(1)}$.


The case for a general number of bodies $P\ge 1$ is now mostly a matter of notation.
The representation is
\be
u = \sum_{k=1}^P \SR\ok \bigl[ (\vec I-\vec L)\bal\ok + \bal_0\ok \bigr],
\label{erepK}
\ee
recalling that $\SR\ok$ is the MFS charge representation from body $k$.
Here the completion flow is constructed with each block of the vector $\vec \alpha_0\in\mathbb R^{NP}$ constant,
\be
\bal_0\ok = \frac{q^{(k)}}{N}\mbf{1}_N, \qquad k=1,\dots,P.
\label{al0}
\ee
The following summarizes the formulation,
and verifies that it solves the elastance problem in the case of exact solution of the linear system.
\begin{pro}[Elastance formulation.]   
\label{p:elast}
Let $\bal\in\R^{PN}$ solve the formally overdetermined linear system with block rows
  \be
     [\vec S^{(kk)}(\vec I-\vec L) + \vec L_r] \bal\ok + \sum_{k'\neq k} \vec S^{(kk')}(\vec I-\vec L)\bal\okp = -\mbf{u}\ok_0
     ,\quad k=1,\dots,P,
  \label{esysK}
  \ee
  where, recalling $\bal_0$ defined by \eqref{al0}, the right-hand side vector has block entries
  \be
  \mbf{u}\ok_0 = \sum_{k'=1}^P \vec S^{(kk')} \bal\okp_0, \qquad k=1,\dots,P.
  \label{erhsK}
  \ee
  Then the potential $u$ given by the representation \eqref{erepK} 
  is harmonic in $\E$, constant
  on the collocation nodes for each body, and has the
  desired net charges \eqref{qkgauss}.
  The constant voltages may be read off as the negative mean strengths
  \be
  \phi^{(k)} = -\frac{1}{N}\sum_{j=1}^N \alpha\ok_j, \qquad k=1,\dots,P.
  \label{vkal}
  \ee
\end{pro}   
\begin{proof}
  Replacing $k$ by $k'$ in the representation \eqref{erepK},
  then evaluating it on the collocation nodes of the $k$th body,
  gives the vector of potentials
  \[
  \{u(\x\ok_i)\}_{i=1}^M =
  \sum_{k'=1}^P \vec S^{(kk')}\bigl[  (\vec I-\vec L)\bal\okp + \bal\okp_0 \bigr],
  \qquad k=1,\dots, P.
  \]
  Subtracting the equation \eqref{esysK} leaves $\{u(\x\ok_i)\}_{i=1}^M = -\vec L_r^{(k)} \bal\ok$,
  which shows that the potential has the same value in all collocation points on the $k$th body,
  and that this constant value is given by \eqref{vkal}. Since it is a sum of fundamental solutions,
  $u$ is harmonic,
  and has the correct charges by construction from \eqref{al0}.
\end{proof}


Note that precomputing the right hand side $\vec u_0 := \{\vec u_0\ok\}_{k=1}^P$
via \eqref{erhsK} can be performed with a single FMM call.
The global system \eqref{esysK} is no less ill-conditioned than \eqref{2sys} for the
Dirichlet problem.

\begin{remark}[Other types of MFS sources.]
The inquisitive reader may wonder: why not instead use double-layer sources for the MFS, obviating the projector $\vec I-\vec L$ and allowing point-like completion sources
as in \cite{PowerMiranda,Pozrikidis1992}?
While this would work, it has two disadvantages over what we propose:
1) it is more cumbersome, requiring two source types, including double-layer sources which are more singular and more expensive to evaluate; and
2) the MFS may also accurately handle close-to-touching lubrication interactions
through admixtures of Stokeslets, rotlets, and doublets \cite{Broms2024}, which generate nonzero force and torque, thus would require the
re-insertion of a projector in any case.
We plan to pursue point 2) for the mobility problem in future work.
\end{remark}

\begin{remark}[Elastance in 2D.]
The 2D case for elastance is more complicated, due to additional constraints
on the total charge, and a more subtle asymptotic form as $|\x|\to\infty$; see \cite{Rachh2016}. We stick to 3D in this work.
\end{remark}

\subsection{Mobility formulation}\label{mobility}

Our proposal is the Stokes generalization of the previous subsection.
The rank-1 single-body projector $\vec L^{(k)}$ now becomes rank-6, projects onto the space of rigid body motions, and is obtained by 
\begin{equation}
    \vec L^{(k)} = \vec K_N^{(k)}\left({\vec K_N^{(k)}}^T\vec K_N^{(k)}\right)^{-1}{\vec K_N^{(k)}}^T.
\end{equation}
We also need a recipe (following \cite{malhotra2023}) for a proxy strength vector $\vec\lambda_0\ok$ that produces the given net forces $\vec f\ok$ and torques $\vec t\ok$ for the $k$th particle.
Analogously to the constant $\bal_0\ok$ given by \eqref{al0} in the elastance case, we restrict
$\vec\lambda_0\ok$ to be in the rigid body space on the proxy points, meaning
\be
\vec\lambda_0\ok = \vec K_N\ok\begin{bmatrix} \vec \nu\ok \\ \vec \xi\ok\end{bmatrix}
\label{lambda01}
\ee
for some rigid body velocity $\vec\nu\ok$ and angular velocity $\vec\xi\ok$.
Inserting this into \eqref{force_syst} gives a $6\times6$ linear system to solve
for each particle,
\begin{equation}
\label{lambda02}
		{\vec K_N\ok}^T\vec K_N\ok
		\begin{bmatrix}
			\vec \nu\ok \\ \vec \xi\ok\end{bmatrix} = \begin{bmatrix} \vec f\ok
   \\ \vec t\ok\end{bmatrix}.
\end{equation}
The resulting $\vec\lambda_0^{(k)}$ in fact solves the minimization problem 
\begin{equation}
    \min\limits_{\vec\mu}\,\|\vec \mu\|^2_2, 
    \text{ s.t.}\medspace{\vec K_N^{(k)}}^T\vec\mu = \begin{bmatrix}
        \vec f^{(k)} \\ \vec t^{(k)}
    \end{bmatrix}.
\end{equation}
With $\vec\lambda_0\ok$ determined for each particle $k$,
the flow representation is
 \begin{equation}\label{multirep}
     \vec u = \sum_{k=1}^P\SR^{(k)}\left[\left(\vec I-\vec L^{(k)}\right)\vec\lambda^{(k)}+\vec\lambda_0^{(k)}\right],
 \end{equation}
 recalling that $\SR\ok$ denotes the $k$th particle MFS Stokeslet representation
 \begin{equation}
     \bigl(\SR^{(k)}\vec\beta\bigr)(\vec x)
     \;\coloneqq\;
     \sum_{j=1}^N\vec G(\vec x,\vec y_j^{(k)})\vec\beta_j,
 \end{equation}
 for any MFS strength vector $\vec\beta$ formed by stacking $\vec\beta_j\in\mathbb R^3$
 for $j=1,\dots,N$.
 Because the projectors $(\vec I-\vec L^{(k)})$ remove net force and torque from each body, 
 \eqref{multirep} generates a flow with the desired net forces and torques.
Enforcing that this flow $\vec u$ is an unknown
rigid body motion on the $k$th particle surface
collocation nodes gives
\be
\sum_{k'=1}^P \vec S^{(kk')}\left[\left(\vec I-\vec L^{(k)}\right)\vec\lambda^{(k')}+\vec\lambda_0^{(k')}\right] \; = \;
\vec K_M\ok \vec U\ok,\qquad
k=1,\dots,P.
\label{stocol}
\ee
By analogy with the elastance case \eqref{ansatz}, we close the linear system via an ansatz that
the $k$th body unknown rigid body motion is controlled by the (thus far unused) rigid body subspace of its Stokeslet coefficients: $\vec U\ok = -{\vec K_N\ok}^T\vec\lambda\ok$.
The right hand side of \eqref{stocol} then becomes $-\vec L_r\ok \vec\lambda\ok$, where
\be
\vec L_r\ok \;:=\;
\vec K_M\ok {\vec K_N\ok}^T
\; \in \R^{3M\times 3N}, \qquad k=1,\dots, P
\label{Lrsto}
\ee
is a rank-6 rectangular rigid body coupling matrix for each particle.

Rearranging \eqref{stocol}, the formulation and its solution of the mobility problem is summarized in the following.
 \begin{pro}[Mobility formulation.]
 Let $\vec\lambda:=\{\vec\lambda\ok\}_{k=1}^P$ solve the tall linear system with block rows 
 \begin{equation}
     \left[\vec S^{(kk)}(\vec I-\vec L^{(k)})+\vec L_r^{(k)}\right]\vec\lambda^{(k)}+\sum\limits_{k\neq k'}\vec S^{(kk')}\left(\vec I-\vec L^{(k)}\right)\vec\lambda^{(k')} = -\vec u_0^{(k)},\quad k = 1,\dots,P,
 \end{equation}
 where, in terms of \eqref{lambda01}--\eqref{lambda02}, the right hand side has block components
 \begin{equation}
     \vec u_0^{(k)} = \sum\limits_{k' = 1}^P\vec S^{(kk')}\vec\lambda_0^{(k')},
     \qquad k=1,\dots,P.
     \label{stou0}
     \end{equation}
     Then, the flow field $\vec u$ given by \eqref{multirep} satisfies the PDE \eqref{stokeseq} in $\mathbb R^3\backslash\overline{\Omega}$, obeys rigid body motion on each particle surface, and corresponds to the desired net forces and torques in \eqref{force_eq} on the particles.  The stack of all rigid body motions may be extracted via
     \begin{equation}
         \vec U = -\vec K_N^T\vec\lambda,
         \label{getU}
     \end{equation}
     with $\vec K_N$ defined in \eqref{Kbig}.
 \end{pro}
  \begin{proof}
      The proof is exactly analogous to that of Proposition~\ref{p:elast}.
  \end{proof}

The above global rectangular linear system is ill-conditioned, as with the resistance MFS formulation, but is amenable to one-body preconditioning, as explained in the next subsection.

 \begin{remark}[Relation to saddle-point systems for mobility.]
A straightforward approach to the mobility problem is to solve $\vec S\vec\lambda = \vec K_M\vec U$ in the least-squares sense, subject to the constraints $\vec K_N\vec\lambda = \vec F$, with unknown rigid body velocities $\vec U$, but known forces and torques $\vec F$.  In saddle point form, this can
for the case $P=1$ be written as
		\begin{equation}\label{saddle_point}
			\begin{bmatrix}
				\vec S^{(11)} & -\vec K_M^{(1)} \\ -{\vec K_N^{(1)}}^T & \vec 0
			\end{bmatrix}\begin{bmatrix} \vec\lambda^{(1)} \\ \vec U^{(1)}  \end{bmatrix}= \begin{bmatrix} \vec 0 \\ -\vec F^{(1)} \end{bmatrix},
		\end{equation}
		 and the mobility matrix $\vec M$ giving $\vec M\vec F^{(1)} = \vec U^{(1)}$ can be solved via the Schur complement $$\vec M = \left({\vec K^{(1)}}^T_N{\vec S^{(11)} }^{+}\vec K^{(1)}_M\right)^{-1}.$$ 
   However, when it comes to the multi-particle $P\gg 1$ case, this does not scale, and we are unaware of a fast algorithm. 
   Yet, the above recompleted formulation has similarities to the null space method in the literature on least squares problems with inequality constraints \cite{Hanson1969,Scott2022} and saddle point problems \cite{Benzi2005}. In a saddle point problem of the form \eqref{saddle_point}, $\vec\lambda^{(1)}$ is then split into two components, one in the null space of ${\vec K_N^{(1)}}^T$ and one in the orthogonal complement, with the goal of solving a system of reduced size. However, that technique, in contrast to ours, is typically beneficial only when the number of constraints is large, and it exploits neither the projectors $\vec L\ok$ nor the relation
\eqref{getU} that we use. 
 \end{remark}

\subsection{One-body preconditioning for elastance and mobility}\label{one_body_revisit}

A dense direct least squares solution of the discrete systems introduced in Sections \ref{elastance} and \ref{mobility} would require $\bigO(P^3N^3)$ effort, which quickly becomes prohibitive as $N$ or $P$ is increased.
The elastance and mobility MFS linear systems may however be right-preconditioned by factorizing
the diagonal blocks, in a similar style to the Dirichlet BVP case from Section \ref{solving_res}, resulting in well-conditioned square linear systems amenable to iterative solution.
We now present the formulae, unifying the elastance and mobility cases. 
Let $\vec S_L^{(k)}$ denote the diagonal matrix block corresponding to particle $k$, that is,
\begin{equation}
    \vec S_L^{(k)} \coloneqq \vec S^{(kk)}(\vec I-\vec L^{(k)}) + \vec L_r^{(k)},
    \qquad k=1,\dots,P.
\end{equation}
For elastance, $\vec L_r^{(k)} = \vec L_r$ and $\vec L^{(k)} = \vec L$ for $k = 1,\dots,P$. The same is true if all particles in a mobility problem have the same shape and orientation; see \eqref{Krot} in Remark \ref{congruent}.
Furthermore, for the elastance problem, $\vec S_L^{(k)}\in \R^{M\times N}$, and  for $k=1,\dots,P$, $\bga\ok = \vec S^{(k)}_L \bal\ok \in \mathbb R^M$. For the mobility problem, $\vec S^{(k)}_L\in  \R^{3M\times 3N}$, and  $\bga\ok = \vec S_L^{(k)}\vec\lambda\ok \in \mathbb R^{3M}$,  for $k=1,\dots,P$.
For a general $P$, the resulting square linear system in both cases takes the form

\begin{equation}\label{2esysprec}
\begin{small}
    \begin{bmatrix}
        \vec I & \vec S^{(12)} \left(\vec I{-}\vec L^{(2)}\right){\vec S_L^{(2)}}^+ & \dots & \vec S^{(1P)} \left(\vec I{-}\vec L^{(P)}\right){\vec S_L^{(P)}}^+  \\
        \vec S^{(21)} \left(\vec I{-}\vec L^{(1)}\right){\vec S_L^{(1)}}^+ & \vec I & \dots & \dots \\
        \vdots & \vdots & \ddots & \vdots \\
        \vec S^{(P1)} \left(\vec I{-}\vec L^{(1)}\right){\vec S_L^{(1)}}^+ & \dots & \dots & \vec I
    \end{bmatrix}
    \begin{bmatrix}
        \vec \gamma^{(1)} \\  \vec \gamma^{(2)} \\ \vdots \\ \vec \gamma^{(P)}
        \end{bmatrix} = 
        \begin{bmatrix}
        -\vec u_0^{(1)} \\  -\vec u_0^{(2)} \\ \vdots \\ -\vec u_0^{(P)}
    \end{bmatrix},
    \end{small}
\end{equation}
recalling the right-hand side definition \eqref{erhsK} for elastance,
or \eqref{stou0} with $\vec\lambda_0\ok$ given by \eqref{lambda01}--\eqref{lambda02} for mobility.
In either case the right-hand side may be computed by a single FMM.
For each body, one first precomputes the two matrices that apply the pseudoinverse 
${\vec S_L^{(k)}}^+$ backward-stably, via the SVD of $\vec S_L^{(k)}$, as in \eqref{pseudo}. For congruent particles, a single SVD is required; rotations determine ${\vec S_L^{(k)}}^+$ from a base discretization as in Remark \ref{congruent}.
Then, one solves \eqref{2esysprec} iteratively for the stacked surface value vector
$\bga \coloneqq \{\bga\ok\}_{k=1}^P$, in analogy with Algorithm \ref{alg:cap} used for the Dirichlet BVPs. There are two sole differences here: 
\begin{enumerate}
    \item the form of the one-body blocks on the diagonal of the system to be preconditioned ($\vec S^{(kk)}$ versus $\vec S^{(k)}_L$), and, 
    \item the application of $\vec I-\vec L^{(k)}$  to each particle's block vector after Step 1 of the \texttt{matvec} function in Algorithm \ref{alg:cap}.
    \end{enumerate}
Upon convergence for elastance,  one recovers $\bal\ok = {\vec S_L^{(k)}}^+\bga\ok$ for each $k$,
finally reporting their averages $\phi^{(k)}$ via \eqref{vkal}. For mobility, $\vec\lambda^{(k)} = {\vec S_L^{(k)}}^+\vec\bga^{(k)}$, $k=1,\dots,P$, and velocities are finally computed via \eqref{getU}. The solution fields in $\mathbb R^3\backslash\overline{\Omega}$ can be recovered via the representations in \eqref{erepK} (for elastance) or \eqref{multirep} (for mobility). 

This completes the description of the preconditioned algorithms, including the option for FMM accelerated iterative solutions.

	\section{Numerical examples}\label{numerics}
 
 We now demonstrate the accuracy and efficiency of the new elastance and mobility formulations, for varying number of objects $P$, separations $\delta$ and discretizations $N$. Iterative solutions use GMRES, and, for the Stokes tests, are accelerated with the multithreaded FMM3D library through a \textsc{matlab} interface, with linear scaling in the number of source points \cite{ChengH.1999,FMM3D}.  The ratio of the smallest to largest singular values of the self-interaction matrices will turn out to be above $\epsilon_{\text{mach}}$ in double precision, so no truncation is needed in the one-body preconditioning schemes of Sections \ref{solving_res} and \ref{one_body_revisit}. 
 

We first study spherical particles in Section \ref{spheres}, where accuracy and the excellent conditioning of both the elastance and mobility solvers are demonstrated.   For Laplace (elastance) we test convergence only for small-scale problems, but for Stokes we also validate against a reference from a boundary integral equation equipped with quadrature by expansion (QBX) \cite{AfKlinteberg2016}. We then move beyond spheres for Stokes in Section \ref{ellipsoids} to consider clusters of many ellipsoids.


\subsection{Spheres}\label{spheres}

In applying the MFS it is convenient to use a quasi-uniform point distribution
on the sphere, both for collocation nodes (radius $R$) and for proxy points (radius $R_p$).
The reason is that the conditioning of the target-from-source matrix $\vec S^{(kk)}$,
already expected to be exponentially bad \cite[Rmk.~1]{Broms2024},
is further worsened by uneven distributions of source and collocation points, as
clustered nodes lead to more linearly dependent columns in the target-from-source matrix.
Following \cite{Broms2024}, we recommend so-called ``spherical design points'' \cite{sphdesign}
  (we have compared this to other sphere point distributions such as Fibonacci grids, but not found any distribution that exceeds spherical designs in efficiency). Note that spherical designs are only available for certain values of N \footnote{Spherical design nodes are available in double-precision accuracy for $N$ up to around 16000 at
		\url{https://web.maths.unsw.edu.au/~rsw/Sphere/EffSphDes/}.}. Throughout the numerical examples for spheres, we choose $M\approx 1.2N$, a ``rectangularity'' typical for MFS \cite{Barnett2007,Liu2016,Stein2022}.

\bfi       
\centering
\begin{subfigure}[t]{0.49\textwidth}
			\centering
					\hspace*{-5ex}\includegraphics[trim = {0cm 0cm 0cm 0cm},clip,width=1.1\textwidth]{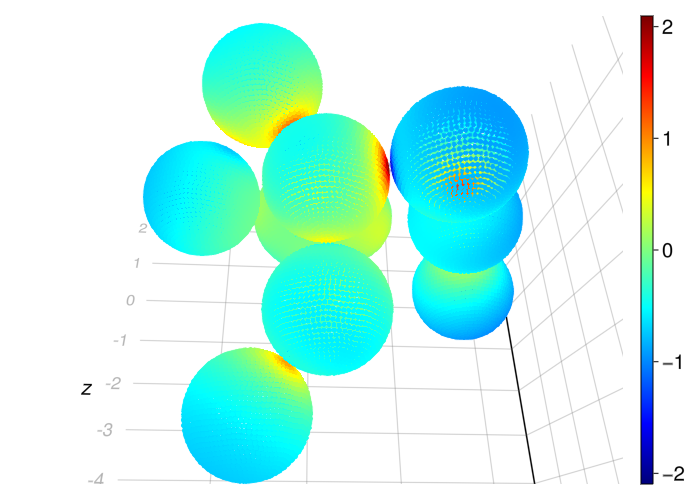}
			\caption{ }		
		\end{subfigure}~
  \begin{subfigure}[t]{0.49\textwidth}
			\centering
   \hspace*{-3ex}
					\includegraphics[trim = {0cm 0cm 0cm 0cm},clip,width=1.1\textwidth]{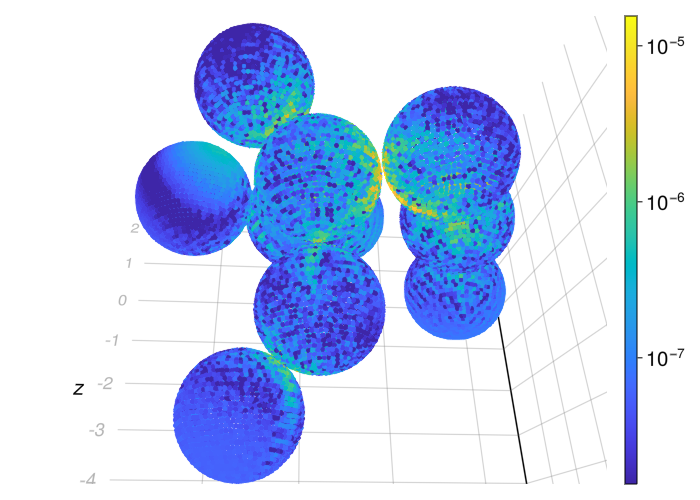}
			\caption{ }
		\end{subfigure}
\ca{Elastance calculation with $P=10$ unit spheres, with many pairs achieving the minimum separation
  $\delta=0.1$. The MFS has $N=969$ unknowns per sphere and $R_p=0.7$. This gives 5 digits of accuracy in the maximum relative boundary condition error (Section \ref{s:elastres}). In panel (a), the charge density $u_n$ is
  shown on a new denser set of $2PM=22460$ surface test points.
  In panel (b), the absolute residual in satisfying constant boundary values is visualized on log color scale, on these same test points.
  }{f:elast}
\efi

\begin{figure}     
\centering
 \begin{subfigure}[t]{0.49\textwidth}
\ig{width=2.65in,trim = {0.5cm 0cm 0.2cm 0.3cm},clip}{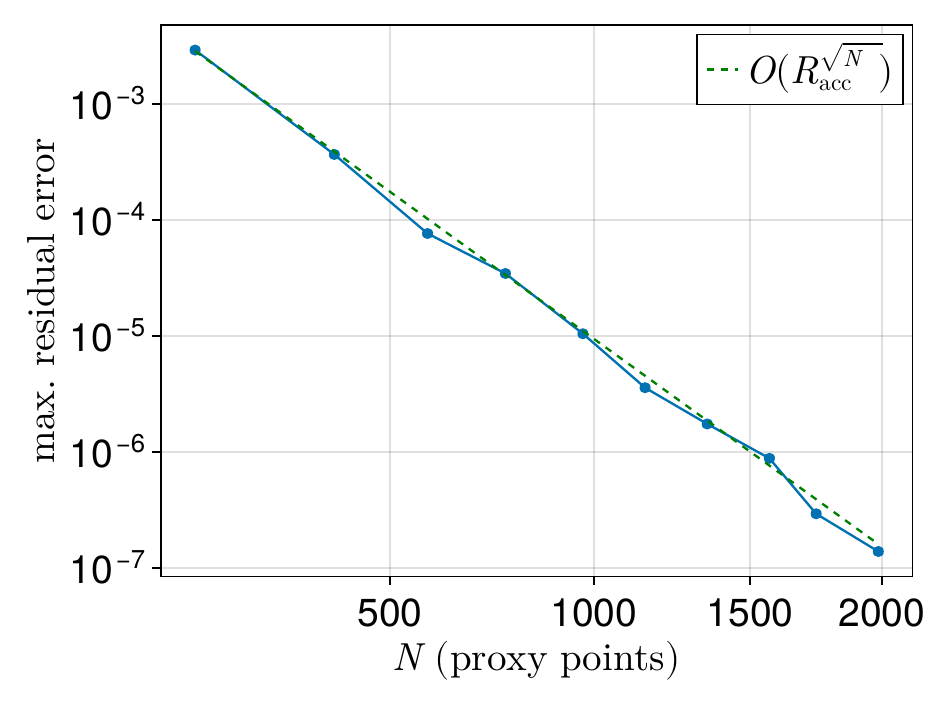} 
\caption{}
\end{subfigure}
 \begin{subfigure}[t]{0.49\textwidth}
\includegraphics[width=2.6in,trim = {0.5cm 0cm 0.3cm 0.3cm},clip]{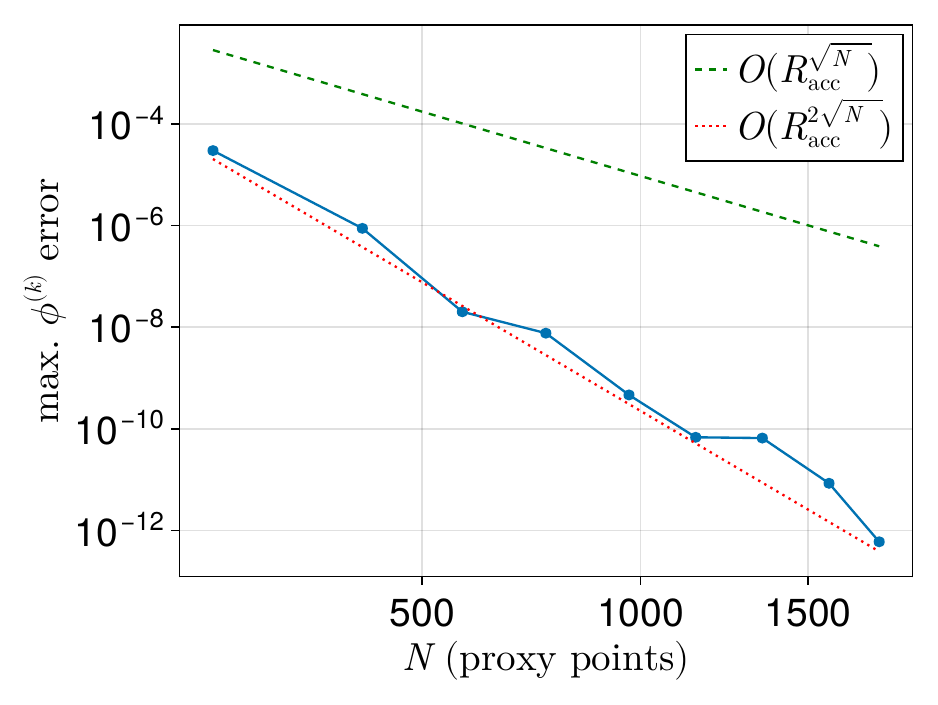}
\caption{}
\end{subfigure}
\caption{Convergence with respect to $N$, the number of proxy-points, for the elastance calculation for the geometry shown in Fig.~\ref{f:elast}.
  Panel (a) shows the maximum residual over the surface
  of all spheres, i.e., $\max_{k=1,\dots,P} \|u-\phi^{(k)}\|_{L^\infty(\pO^{(k)})}$.
Panel (b) shows convergence of the maximum error in the computed potentials
$\phi^{(k)}$ (estimated using converged results computed at $N=1986$ as a reference).  
  The number of collocation nodes was chosen as $M\approx 1.2 N$, fixing $R_p=0.7$.
  Root-exponential convergence is indicated as a dashed straight line (our horizontal axis being linear in $\sqrt{N}$), for both errors.
  The rate for $\phi^{(k)}$ in panel (b) is close to twice the rate in panel (a).
  See Section \ref{s:elastres} for an explanation of $R_\text{acc}$.
  }
  \label{f:elastconv}
\end{figure}

 \subsubsection{Laplace: small scale elastance convergence test}
\label{s:elastres}
 
We test the MFS elastance formulation of Section~\ref{elastance}
with one-body preconditioning as in Section~\ref{one_body_revisit}, for a random cluster of $P=10$ unit spheres ($R=1$) with a given minimum separation of $\delta$.
The cluster is ``grown'' from the origin, with every new particle in a randomly sampled direction at the minimum distance from the origin such that the minimum distance to at least one other sphere is exactly $\delta$.

We study convergence (accuracy vs $N$) at a moderate separation $\delta=0.1$.
This is not the smallest $\delta$ that the proposed scheme is practical for:
separations down to $\delta=0.01$ may be handled to several digits of accuracy in the scalar case, but would demand $N>5000$ and high resulting SVD costs (tens of seconds).

Fixing $R_p=0.7$, with $N=969$ proxy points,
the
resulting surface charge density and magnitude of the surface residual are shown in Fig.~\ref{f:elast}.
A couple of implementation details must be explained.
The ``bare'' full system matrix comprising all offdiagonal ($k\neq k'$) blocks
$\vec S^{(kk')}$, $k,k'=1,\dots,P$, was filled once (size $PM\times PN$, i.e., $11230 \times 9690$), enabling the matrix-vector multiply with the
preconditioned system matrix \eqref{2esysprec} to be performed in four stages:
i) apply the block pseudoinverses ${\vec S^{(k)}}^+$ to the blocks of the vector as in \eqref{pseudo}, ii) apply $(\vec I-\vec L)$ to the blocks of the vector,
iii) multiply this vector by the bare offdiagonal system matrix,
and iv) add this to the original vector (for the identity term).
With this matrix-vector multiply,
GMRES was used to solve \eqref{2esysprec} iteratively, with a relative stopping tolerance $10^{-8}$.
The proxy strengths $\bal^{(k)} = {\vec S^{(k)}}^+ \bga^{(k)}$ were recovered from the solution vector,
and the voltages $\phi^{(k)}$ extracted via \eqref{vkal}.
This whole solution takes 3 seconds of CPU time.
Around 5 uniform digits of surface residual is achieved;
it is clear from the figure that the worst residuals occur at the close-to-touching areas.

The convergence of the maximum surface residual error vs $N$ is shown in Fig.~\ref{f:elastconv}a:
it is remarkably consistent with root-exponential (dashed line), down to 7-digit
uniform surface residual error.
Here $15$ GMRES iterations were needed for every $N$ tested in the graph, indicating
a small and stable condition number.

Yet, with a little theory, one can go further and explain the {\em rate} of
root-exponential convergence, as follows (we first reported this for the Stokes resistance problem in \cite[Sec.~3]{Broms2024}).
Recall that the space of spherical harmonics up to degree $p$ has dimension $(p+1)^2$.
With one proxy point per spherical harmonic%
\footnote{Since spherical designs integrate spherical harmonics of degree $p$ exactly using $N\sim p^2/2$ points \cite{womersley18}, one might expect two proxy points needed per harmonic; however that rate fits the data less well, a fact that we leave for future investigation.}
we would thus expect a maximum degree $p\approx\sqrt{N}$.
Standard MFS theory (by analogy with 2D rigorous analysis \cite{Ka89}) predicts
that convergence should be geometric in $p$ with a rate controlled by
the minimum radius $R_\text{acc}<1$ to which the solution may be continued {\em inside} the sphere as a regular Laplace solution. For a pair of unit spheres separated by $\delta$ this is known:
\[
R_\text{acc} = 1 + \delta/2 - \sqrt{\delta +\delta^2/4}
\]
is the accumulation point radius for the multiple image reflection series \cite{Cheng2000,Broms2024}. At $\delta=0.1$ one gets $R_{\text{acc}}\approx0.73$, and the rate fit with the resulting $R_{\text{acc}}^{\sqrt{N}}$ in Fig.~\ref{f:elastconv}a appears excellent.

\begin{remark}[MFS convergence rate regimes.]\label{r:rate}
By analogy with rigorous MFS theory in 2D \cite{Ka89,Barnett2007}, we expect two regimes for the error convergence: 1) when $R_\text{acc}>R_p^2$ the
singularity in the continuation (decay of on-surface spherical harmonic coefficients) dominates to give a rate $\bigO(R_\text{acc}^p)$; otherwise 2) the discreteness (aliasing error) of the proxy points takes over, giving $\bigO(R_p^{2p})$.   
Our parameters fall into the first regime. Indeed,
empirically we find that changing $R_p$ from $0.7$ to $0.8$ (not shown) has no effect on the rate (although the error prefactor worsens).
At our $\delta=0.1$, only when $R_p > \sqrt{R_\text{acc}} \approx 0.85$
would we enter the second, aliasing-dominated regime. The MFS convergence regimes are also visualized in Fig.~\ref{conv_pair} of Appendix \ref{acc_spheres} in the Stokes case.
\end{remark}
 
Fig.~\ref{f:elastconv}b shows the convergence for the potentials
$\phi^{(k)}$ themselves, which appears to have nearly twice the root-exponential rate of the surface residual maximum error (the dotted red line shows twice the rate for comparison). An improvement of rate (but only by a factor 1.5) for such body-averaged quantities was also found in the resistance problem \cite[Fig.~4]{Broms2024}.
By around $N=1700$ per sphere, 12 digits are achieved for the absolute error in the potentials.

	\subsubsection{Stokes: comparison to a boundary integral reference}\label{BIEcomp}
 
A useful numerical test is to solve the resistance and mobility problems in succession,
since they are supposed to be each other's inverse.
This starts with a known vector $\vec U_{\text{ref}}\in\mathbb R^{6P}$ of stacked translational and angular velocities, and
computes $\vec F\in\mathbb R^{6P}$, a stacked vector of the net forces and torques on all particles, by solving the resistance problem. We then insert $\vec F$ as the input to the mobility solver to get an output vector $\vec U$, and
finally report 
\begin{equation}\label{2way}    
\epsilon_{\text{2-way}} \coloneqq \|\vec U-\vec U_{\text{ref}}\|_\infty/\|\vec U_{\text{ref}}\|_\infty
\end{equation}
as an error metric. We call this the {\em 2-way error}. When the 2-way error is quantified here and later, distinct proxy radii are used for resistance and mobility to avoid misleading cancellations (``inverse crimes''): specifically, the offsets for the proxy-surface in the normal direction from the particle surface are set such that $\Delta_{\text{sep}}^{\text{mob}} = 1.05\Delta_{\text{sep}}^{\text{res}}$.  
In the sphere examples, we will report on the radius of the resistance problem,  $R_p^{\text{res}} \coloneqq 1-\Delta_{\text{sep}}^{\text{res}}$. 

One first should validate the 2-way error against an independent error metric. 
For the latter we compare our mobility solve particle velocity outputs to those from an established BIE reference method using QBX quadrature \cite{AfKlinteberg2016} with 900 unknowns per body and a quadrature error tolerance of $10^{-6}$, and report
\begin{equation}\label{BIEerr}
    \epsilon_{\text{BIE}}\coloneqq \|\vec U-\vec U_{\text{BIE}}\|_\infty/\|\vec U_{\text{BIE}}\|_\infty.
\end{equation}
The validation results are in Fig.~\ref{Fig:sphere_QBX}.
We chose 56 test geometries consisting of five randomly positioned unit spheres driven by randomly sampled forces and torques. The BIE solution is much more costly to compute than the MFS solution, due to the quadrature setup.
The radius of the proxy-surface is set to $R_p = 0.59$ and $N=686$.
We report $\epsilon_{\text{BIE}}$ in Fig.~\ref{Fig:sphere_QBXa} versus the minimum particle-particle distance $\delta$ in each geometry. For small $\delta$, the interaction becomes increasingly difficult to resolve, both with MFS and the BIE reference, due to
lubrication effects between particles under relative motion.
However, Fig.~\ref{Fig:sphere_QBXb} demonstrates that $\epsilon_{\text{2-way}}$ is a faithful surrogate for $\epsilon_{\text{BIE}}$, at least within $\pm1$ digit of accuracy,
in the range above the MFS GMRES tolerance of $10^{-7}$. 
Thus, $\epsilon_{\text{2-way}}$ alone will be used as the error metric from now on. 
	\begin{figure}[t!]
		\centering
  \begin{subfigure}[t]{0.49\textwidth}
			\centering
					\includegraphics[trim = {5cm 16.8cm 9.0cm 4.4cm},clip,width=0.8\textwidth]{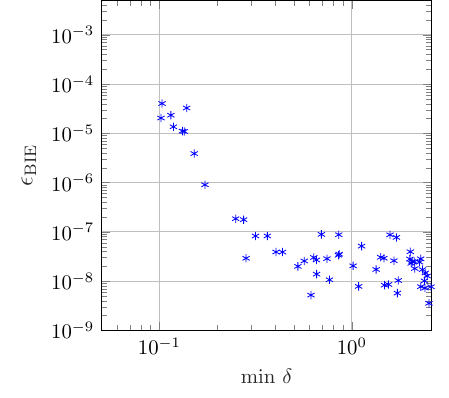}
			\caption{Error relative to the BIE reference, as function of the minimum separation $\delta$ between the particles. }
			\label{Fig:sphere_QBXa}
		\end{subfigure}~~~~
		\begin{subfigure}[t]{0.49\textwidth}
			\centering
					\includegraphics[trim = {5cm 16.8cm 9.0cm 4.4cm},clip,width=0.8\textwidth]{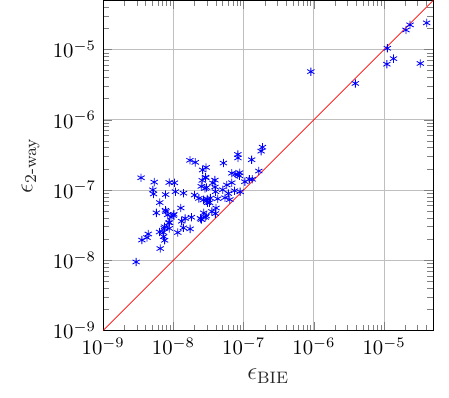}
			\caption{Error relative to the BIE reference compared to the 2-way error.} 
			\label{Fig:sphere_QBXb}
		\end{subfigure}~
		\caption{We validate the error metric $\epsilon_{\text{2-way}}$ defined by \eqref{2way} against the BIE reference error $\epsilon_{\text{BIE}}$ defined by \eqref{BIEerr}, for geometries with five random spheres. In panel (b), points above the red line---a majority---represent configurations where the 2-way error $\epsilon_{\text{2-way}}$ is the largest measure of the error. 
  The GMRES tolerances for the reference BIE and MFS solutions are set to $10^{-6}$ and $10^{-7}$ respectively.} 
		\label{Fig:sphere_QBX}
	\end{figure}
 

 \subsubsection{Stokes: clusters of spheres}
 
	Next, we test accuracy for varying minimum separation distances $\delta$ and particle numbers $P$. To support the claim of a well-conditioned solver, we report GMRES iteration counts for the one-body preconditioned schemes of Section \ref{solving_res} for the resistance problem, and Section \ref{one_body_revisit} for the mobility problem. We consider clusters of spheres of
 the same type as in Example \ref{s:elastres}. Fig.~\ref{sphere_cluster} shows an example, with the traction vector computed at 2000 new points per particle.
 A few fixed choices of proxy point number $N$ per particle, and proxy radius $R_p$, are tested. 
 We report $\epsilon_{\text{2-way}}$ defined in \eqref{2way}, with 
 randomly sampled rigid body velocity inputs.
	
	\begin{figure}[h!]
		\centering
  \begin{subfigure}{0.49\textwidth}
  \hspace*{-3ex}
		\includegraphics[trim = {1cm 1cm 0cm 0cm},clip,width=1.1\textwidth]{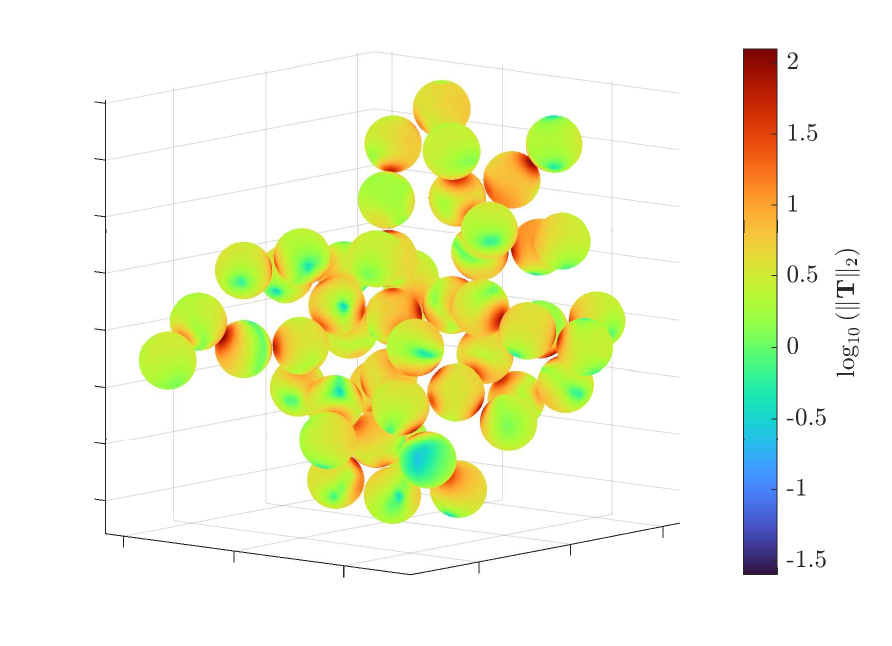}
  \caption{}
  \end{subfigure}
  \begin{subfigure}{0.5\textwidth}
  \centering
\includegraphics[trim = {1cm 1cm 0cm 0cm},clip,width=1.1\textwidth]{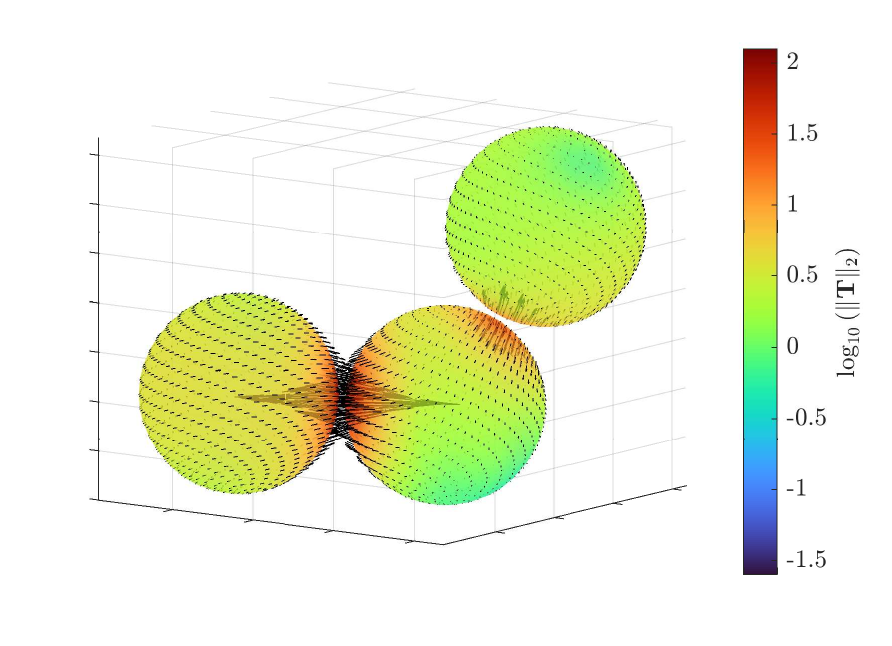}
      \caption{}
  \end{subfigure}
  		\caption{Point-wise force density magnitude $\|\vec T\|_2$, where traction $\vec T$ is defined by \eqref{traction}, computed on: (a) an example cluster of 50 spheres with random rigid body motions and $\delta = 0.1$, and (b) a smaller example of the same kind, with also the direction of the traction $\vec T$ displayed with black arrows.}	
  \label{sphere_cluster}
	\end{figure}
	
	\begin{example}[Performance for decreasing particle gaps.]\label{delta_dep}
	The number of GMRES iterations needed to solve the  mobility and resistance problems are reported as functions of $\delta$ for clusters of 10 spheres in Figs.~\ref{Gmres_mob} and \ref{Gmres_res}. Considerably fewer iterations are required for the mobility problem than for the resistance problem for all $\delta$ and combinations $(N,R_p)$. 
 The 2-way error for the same test is reported  in Fig.~\ref{err_10cluster}, with shaded regions representing the maximum and minimum error from solving the problem 10 times for each $\delta$, with different configurations and velocity data.
 For two of the reported discretization pairs $(N,R_p)$, the error level is less than $10^{-3}$ for all investigated separations, i.e., down to $\delta = 0.05$. (The larger error observed with $R_p = 0.7$ and $N=686$ is simply explained by convergence regime 1 in Remark \ref{r:rate}.)
 
	\begin{figure}[t]
		\centering	
  \begin{subfigure}[t]{0.32\textwidth}
			\centering
			\includegraphics[trim = {5.3cm 16.2cm 7.5cm 4.4cm},clip,width=1.1\textwidth]{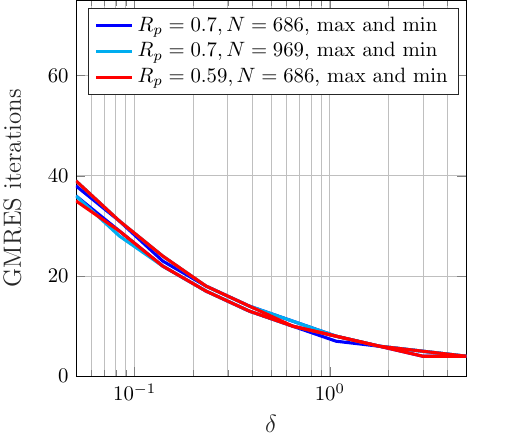}
			\caption{Mobility problem iteration count as function of particle separation $\delta$.}
			\label{Gmres_mob}
		\end{subfigure}~
		\begin{subfigure}[t]{0.32\textwidth}
			\centering
			\includegraphics[trim = {5.3cm 16.2cm 7.5cm 4.4cm},clip,width=1.1\textwidth]{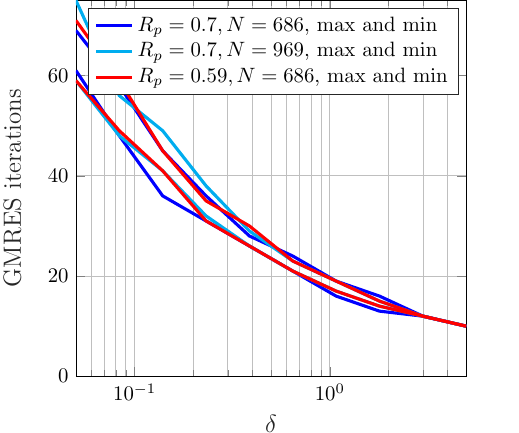}
			\caption{Resistance problem iteration count as function of $\delta$.}
			\label{Gmres_res}
		\end{subfigure}
  		\begin{subfigure}[t]{0.32\textwidth}
			\centering
			\includegraphics[trim = {5.3cm 16.2cm 7.5cm 4.2cm},clip,width=1.1\textwidth]{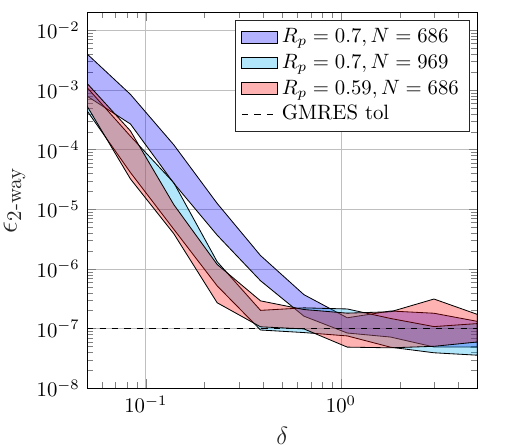}
			\caption{2-way error as function of $\delta$.}
			\label{err_10cluster}
		\end{subfigure}
		\caption{Performance of the resistance and mobility formulations for clusters of 10 spheres grown for variable particle-particle separations $\delta$. In panel (a) and (b), the number of GMRES iterations are displayed for the mobility and resistance problems, 
  with smaller iteration counts observed for mobility. The resulting 2-way error range is reported vs $\delta$ in panel (c). 
  For each $\delta$, the experiment is repeated in 10 different runs, with randomly sampled clusters and rigid body velocities for each. Results in panels (a), (b) and (c) display max and min for the 10 runs for each choice of $R_p$ and $N$ used for discretizing the proxy-surface. }
		\label{fig_10cluster}
	\end{figure}
 \end{example}
	
	\begin{example}[Performance with increasing number of particles.]\label{vary_P}
Next, we vary the number of particles $P$ in the cluster and quantify the number of GMRES iterations for fixed choices of $\delta$ in Fig.~\ref{Fig:varyP}. Remarkably, for the mobility problem, the iteration count does not grow with $P$, while for the resistance problem the iteration count grows weakly (roughly as $\sqrt{P}+c$ for some constant $c$). \footnote{This big convergence rate difference between resistance and mobility is in agreement with other numerical methods in the literature; see a discussion on p.251 of \cite{USABIAGA2016}.} For the two separations $\delta = 0.2$ and $\delta = 1$, the 2-way error and the maximum magnitudes of the coefficient vector $\vec\lambda$ are also displayed for both the resistance and mobility problems,  in Figs.~\ref{Fig:errorP} and \ref{Fig:magnP}. For smaller $\delta$, the problem is harder to resolve and the magnitude of the coefficient vector is increased, as expected for the MFS \cite{Doicu2000,Barnett2007}.
	\end{example}
	\begin{figure}[t]
		\centering	
  \begin{subfigure}[t]{0.32\textwidth}
			\centering
            \hspace*{-3ex}
			\includegraphics[trim = {5.3cm 15.7cm 7.5cm 4.2cm},clip,width=1.1\textwidth]{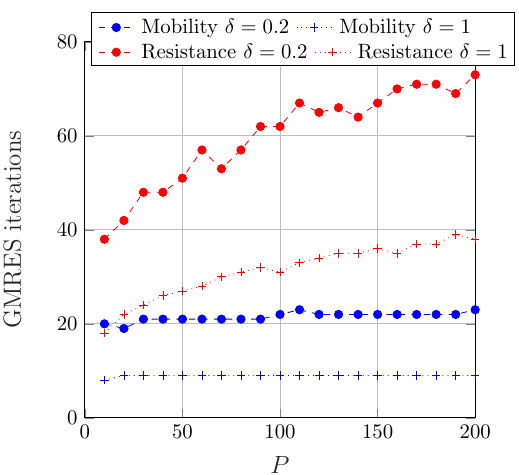}
			\caption{}
			\label{Fig:iterationP}
		\end{subfigure}~~
		\begin{subfigure}[t]{0.32\textwidth}
			\centering
            \hspace*{-3ex}
			\includegraphics[trim = {5.3cm 15.9cm 7.5cm 4.2cm},clip,width=1.1\textwidth]{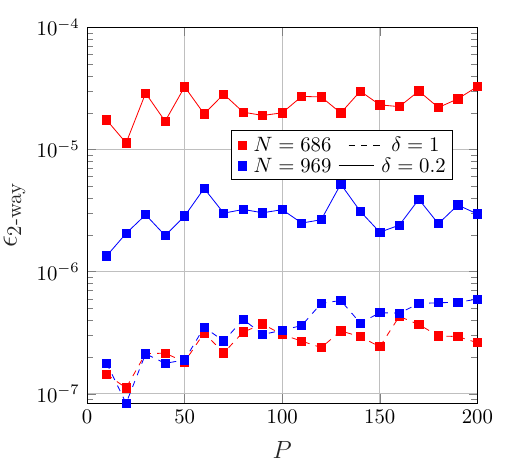}
			\caption{}
			\label{Fig:errorP}
		\end{subfigure}
		\begin{subfigure}[t]{0.32\textwidth}
			\centering
            \hspace*{-1ex}
			\includegraphics[trim = {5.3cm 15.9cm 7.5cm 4.2cm},clip,width=1.1\textwidth]{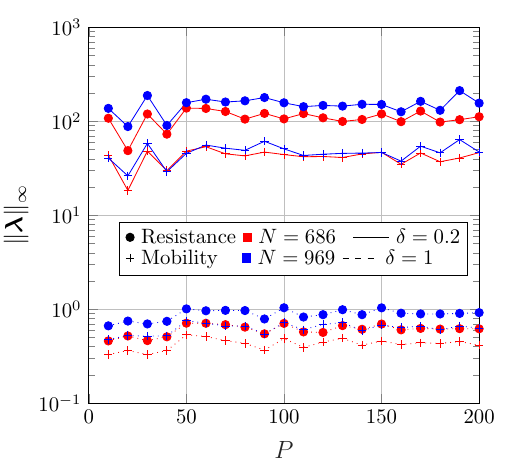}
			\caption{}
			\label{Fig:magnP}
		\end{subfigure}
		\caption{Performance for constant-density clusters of $P$ spheres. In panel (a), it is shown that the number of GMRES iterations grows weakly with $P$ for the resistance problem, but stays constant for the mobility problem. As can be seen in panel (b), the 2-way error depends both on $N$ (proxy point number per particle), and on the minimum separation $\delta$, but not on $P$. In panel (c), $\|\vec\lambda\|_{\infty}$ on any particle stays constant in both the mobility and resistance problems with increasing $P$, but is slightly larger for the resistance problem than for mobility. 
  (This could however be an effect of the slightly smaller proxy-radius chosen for the resistance problems.) Each particle travels with randomly sampled rigid body velocities and the resistance and mobility problems are solved in succession. The GMRES tolerance is set to $10^{-7}$, and $R_p^{\text{res}}=0.7$ for both choices of $N$. For $P>50$, FMM is used with tolerance set to $10^{-8}$.}
		\label{Fig:varyP}
	\end{figure}
 
 \begin{remark}[FMM tolerance.]\label{FMM_tol}
     The FMM tolerance has to be sufficiently strict so as not to affect the error level and iteration count for large $P$. The reason is that the surface velocity for a single particle determined with FMM does not exactly match the subtracted self-interaction computed with direct summation in each matrix-vector multiply, in the iterative solution schemes of Sections \eqref{solving_res} or \eqref{one_body_revisit}. A limiting factor is the {\em magnitude} of the coefficient vector, $\|\vec\lambda\|_{\infty}$, which serves to amplify the FMM error.
     For instance, since $\|\vec\lambda\|_{\infty}\approx 10^2$ in Fig.~\ref{Fig:magnP}, we set an FMM tolerance of $10^{-8}$ to achieve 6-digit evaluation error.
 \end{remark}

	\subsection{Ellipsoids}\label{ellipsoids}
 Next, we demonstrate the capability of our scheme to handle smooth non-spherical particles. 
 	We consider large clusters of two types of ellipsoids: spheroids with semiaxes $a=b=0.5$ and $c= 1$ (``Type S''), or triaxial ellipsoids with $a=0.4$, $b=0.6$ and $c= 1$ (``Type T''). In Examples \ref{conv_ellipsoid} (a convergence test) and \ref{large_el_ex} (a large scale demonstration), we grow clusters of ellipsoids of Type S or T with each particle $\delta$ away from at least one neighbor, using alternating projection to compute pair separations. Examples of clusters of ellipsoids of Type T are illustrated in Figs.~\ref{large_ex} and \ref{Fig:Ellipsoids}. Apart from the 2-way error, we also report the {\em relative surface residual} defined pointwise by
 \begin{equation}\label{residual}
 \epsilon_{\text{res}}(\vec x)\coloneqq \|\vec u(\vec x)-\vec g(\vec x)\|_{2}/\|\vec g(\vec x)\|_{2}, \qquad \vec x \in \partial\Omega,
 \end{equation}
 with $\|\cdot\|_2$ the Euclidean norm in $\mathbb R^3$, $\vec u$ the solution flow field given by the MFS representation, and $\vec g$ the rigid-body surface velocity \eqref{rbm} using the computed particle velocities.
 This is simply the Stokes analog of the residual shown in Figure~\ref{f:elast}b.
 Its maximum is estimated using a large set of points on all particle surfaces. 

Both the proxy and collocation surfaces are discretized with the quasi-uniform ellipsoid grid described by Stein and the 2nd author in \cite{Stein2022}. In brief, an ellipsoid is parameterized
as $(a\sqrt{1-t^2}\cos{s},b\sqrt{1-t^2}\sin{s},ct)$
in Cartesians relative to its center, where $(s,t)\in[0,2\pi]\times[-1,1]$.
These parameters are discretized with $N_v$ Gauss--Legendre nodes in the $t$-direction, then periodic trapezoidal nodes in the $s$-direction. The number of the latter are roughly bounded by $0.75N_v$, and varied with $t$ to give an approximately uniform surface density.
The number of nodes 
on the proxy-surface, normal-shifted from the true surface, is then $N\approx17+3.9N_v+0.44N_v^2$. 
  For the collocation surface, we increase the number of Gauss-Legendre nodes such that $M \approx17+3.9(1.15N_v)+0.44(1.15N_v)^2$, resulting in $M\approx1.3N$. The rectangularity $M/N$ must be larger than for spheres, we believe due to the greater variation in node uniformity and curvature. However, this does not affect the dominant cost, which is the FMM scaling only with $N$.
   
  \begin{example}[Convergence study.]\label{conv_ellipsoid}  For four ellipsoids of Type T with given rigid body motions, and separations $\delta = 0.5$, the error dependence on $N_v$, and the separation between proxy and collocation surfaces $\Delta_{\text{sep}}$, is studied in Fig.~\ref{Fig:deltasep}. This was then used to pick a good $\Delta_{\text{sep}}$ for later tests.
Much as in the elastance sphere case, Fig.~\ref{Fig:convergenceN}
shows that the convergence of the maximum residual is slower than that of the 2-way error,
but that both of them are spectrally accurate with respect to $N_v$.

 \begin{figure}[t]
 \centering
 		\begin{subfigure}[t]{0.45\textwidth}
			\centering
			\includegraphics[trim = {5.3cm 16cm 5.8cm 4.5cm},clip,width=1.12\textwidth]{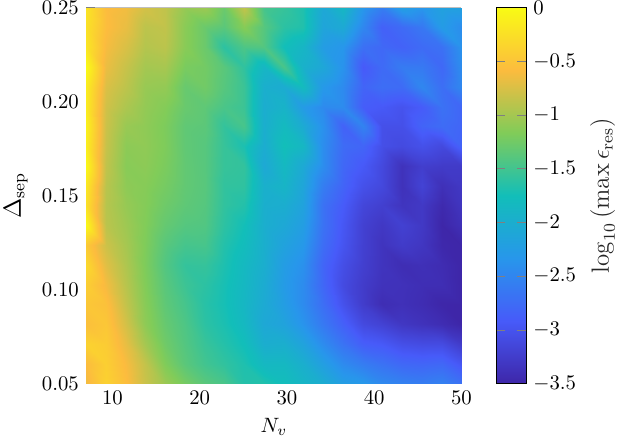}
			\caption{}
			\label{Fig:deltasep}
		\end{subfigure}~~~~~
		\begin{subfigure}[t]{0.44\textwidth}
			\centering
			\includegraphics[trim = {5.3cm 16cm 8.0cm 4.5cm},clip,width=0.9\textwidth]{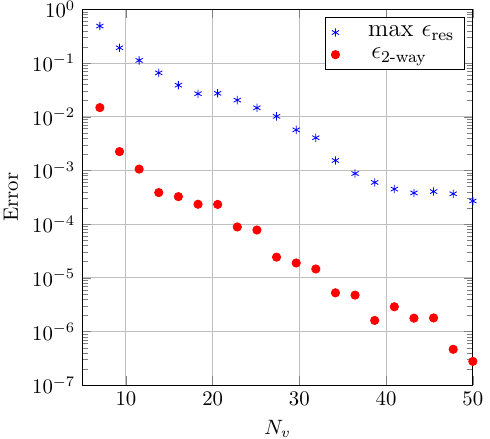}
			\caption{}
			\label{Fig:convergenceN}
		\end{subfigure}
		\caption{Convergence test for ellipsoids with a geometry of  four particles of Type T (with semiaxes $\lbrace 0.4, 0.6, 1\rbrace$) and separation $\delta = 0.5$.  In panel (a), the distance between the proxy and collocation surfaces, $\Delta_{\text{sep}}$, is varied together with the  particle resolution, as set by $N_v$, and the maximum relative residual is displayed. In panel (b), convergence in $N_v$ is shown for two error metrics, with $\Delta_{\text{sep}}^{\text{res}} = 0.125$. }
		\label{Fig:convergenceE}
	\end{figure}
 \end{example}

 \begin{example}[Large scale tests.]\label{large_el_ex}
 Finally, we consider larger clusters of ellipsoids of Types S and T, with varying number of particles $P$ and separations $\delta$. For each configuration, the performance for both resistance and mobility solves is quantified in terms of GMRES iterations, total CPU solution times, coefficient magnitudes $\|\vec\lambda\|_{\infty}$, and the two error metrics $\epsilon_{\text{2-way}}$ and $\max\,\epsilon_{\text{res}}$.  Results are summarized in Table \ref{Tab:ellipsoids}. The reported solution times 
 correspond to roughly linear scaling in the number of particles and are dominated ($>90\%$) by the FMM -- more so for larger $P$ than for smaller. 
 The number of iterations as functions of $P$ and $\delta$ follow the same trends as for spheres: there is no increase in iterations with $P$ for the mobility problem, while there is a growth for the resistance problem. For both problems, a larger iteration count is required as $\delta$ decreases. 
 Similar performance is observed for the two ellipsoid aspect ratios, but problems with particles of Type T are slightly harder to resolve.  
    In Fig.~\ref{Fig:resiudal}, $\epsilon_{\text{res}}$  is shown over the surfaces of a subset of the particles in a larger simulation. The largest residuals are obtained on parts of a particle surface of high curvature or close to touching another body. The corresponding force density on the proxy-surfaces in the resistance solve is visualized in Fig.~\ref{Fig:density}, and is typically of larger magnitude towards the particle tips.  
 For both the mobility and resistance problems, $N_v = 40$, such that $N=864$, $M = 1124$. 
 For the mobility problem, the relative residual $\epsilon_{\text{res}}$ is measured at 1732 points per particle surface using the computed rigid body velocities.
 A key conclusion is that the 2-way error is consistently 2 digits more accurate than the maximum residual error, and still retains 2-3 accurate digits even at the close separation $\delta=0.05$.

	\begin{figure}[h!]
  \centering
  \hspace*{-2ex}
  

		\begin{subfigure}[b]{0.49\textwidth}
   \begin{overpic}[trim = {1.4cm 1cm 0.5cm 0cm},clip,width=\textwidth]{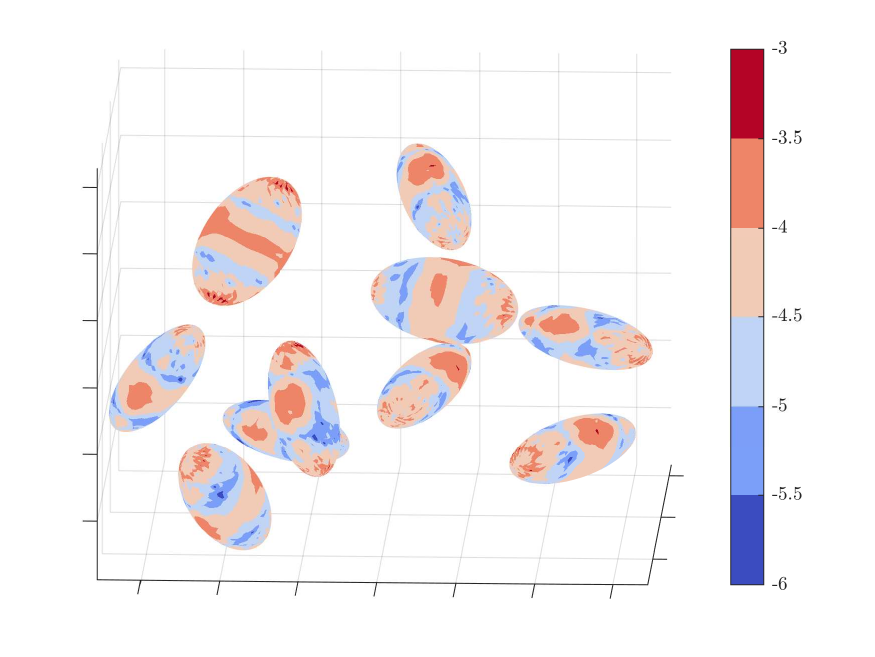}
\put(96,19){\rotatebox{90}{\colorbox{white}{\parbox{0.4\linewidth}{%
     \small{$\qquad\log_{10}\left(\epsilon_{\text{res}}\right)$}}}}}
     \end{overpic}
				\caption{Relative residual on particle surfaces.}
				\label{Fig:resiudal}
    \end{subfigure}~~
    \begin{subfigure}[b]{0.49\textwidth}
   \begin{overpic}[trim = {1.6cm 1cm 0.5cm 0cm},clip,width=1.05\textwidth]{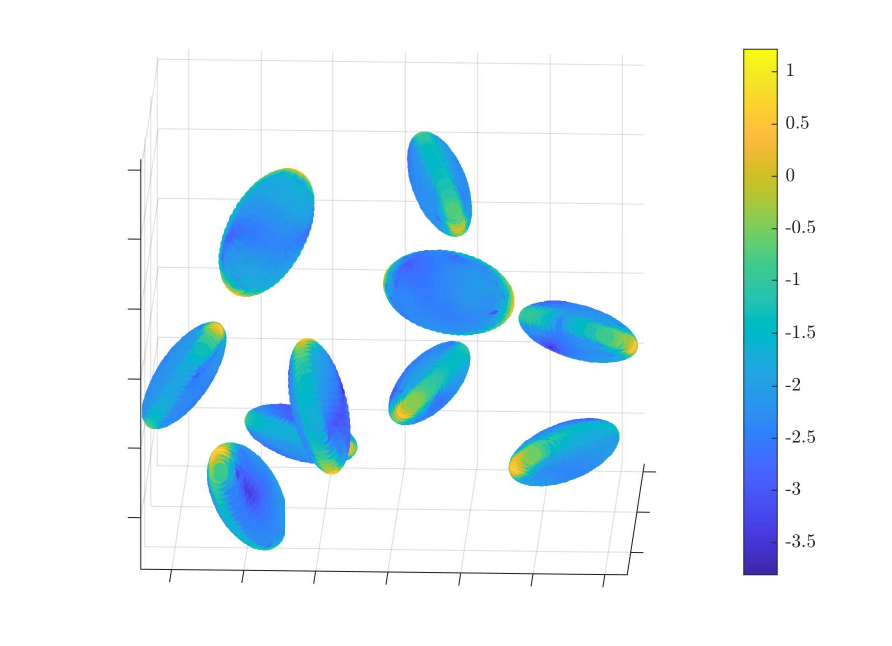}
\put(96,19){\rotatebox{90}{\colorbox{white}{\parbox{0.4\linewidth}{%
     \small{$\qquad\log_{10}\left(\|\vec\lambda\|_2\right)$}}}}}
     \end{overpic}
				\caption{Force density on proxy-surfaces.}
				\label{Fig:density}
		\end{subfigure}
		\caption{A subset of a large cluster of type T ellipsoids with semiaxes $\lbrace 0.4, 0.6, 1\rbrace$ and $\delta = 0.5$.}
 		\label{Fig:Ellipsoids}
 	\end{figure}
	\begin{table}[h!]
		\centering
		\begin{tabular}{l|c|c|c|c|c|c|c|c}
			\multirow{2}{*}{\parbox{1.4cm}{Type, $P$, $\delta$}} & \multirow{2}{*}{\parbox{0.8cm}{$\max\,\epsilon_{\text{res}}$}} & \multirow{2}{*}{$\epsilon_{\text{2-way}}$} & \multirow{2}{*}{\parbox{0.5cm}{iters. mob.}} & \multirow{2}{*}{\parbox{0.5cm}{iters. res.}} & \multirow{2}{*}{\parbox{0.4cm}{$t_{\text{mob}}$ [s]}} & \multirow{2}{*}{\parbox{0.4cm}{$t_{\text{res}}$ [s]}} & \multirow{2}{*}{\parbox{0.5cm}{$\|\vec\lambda\|_{\infty}$ mob.}} & \multirow{2}{*}{\parbox{0.5cm}{$\|\vec\lambda\|_{\infty}$ res.}} \\ & & & & & & &\\ \hline
\hline 
S, 100, 0.5 & $1.04\times 10^{-3}$ & $2.99\times 10^{-6}$ & 10 & 36 & 114 & 317 & $8.1\times 10^{1}$ & $1.2\times 10^{2}$ \\ \hline
S, 500, 0.5 & $9.55\times 10^{-4}$ & $3.54\times 10^{-6}$ & 10 & 50 & 468 & 1932 & $1.4\times 10^{2}$ & $1.9\times 10^{2}$ \\ \hline
T, 100, 0.5 & $4.38\times 10^{-3}$ & $1.88\times 10^{-5}$ & 10 & 36 & 113 & 324 & $1.5\times 10^{3}$ & $4.5\times 10^{3}$ \\ \hline
T, 500, 0.5 & $3.94\times 10^{-3}$ & $2.18\times 10^{-5}$ & 10 & 50 & 460 & 1905 & $1.7\times 10^{3}$ & $4.9\times 10^{3}$ \\ \hline
S, 100, 0.1 & $6.69\times 10^{-2}$ & $7.83\times 10^{-4}$ & 24 & 72 & 217 & 560 & $8.5\times 10^{3}$ & $1.4\times 10^{4}$ \\ \hline
S, 100, 0.05 & $2.64\times 10^{-1}$ & $3.38\times 10^{-3}$ & 35 & 93 & 305 & 733 & $3.9\times 10^{4}$ & $6.3\times 10^{4}$ \\ \hline
T, 100, 0.1 & $1.22\times 10^{-1}$ & $1.06\times 10^{-3}$ & 25 & 73 & 217 & 577 & $1.5\times 10^{4}$ & $4.6\times 10^{4}$ \\ \hline
T, 100, 0.05 & $4.03\times 10^{-1}$ & $5.60\times 10^{-3}$ & 38 & 101 & 322 & 803 & $7.4\times 10^{4}$ & $1.4\times 10^{5}$ \\ 
\hline 
		\end{tabular}
  	\caption{Clusters of $P$ ellipsoids at least $\delta$ apart are studied for the two choices $P=100$ and $P=500$. Type S are spheroids with semiaxes $\lbrace 0.5,0.5,1\rbrace$ and Type T are triaxial ellipsoids with  semiaxes $\lbrace 0.4, 0.6, 1\rbrace$. The total solution times in $t_{\text{res}}$ and $t_{\text{mob}}$ indicate linear scaling in $P$. The one-body SVD  takes 6 seconds of CPU time.  
  Smaller errors are obtained with the particles of Type S, for which the maximum MFS coefficients also are smaller.\label{Tab:ellipsoids}}
	\end{table}

For the even larger
example in Fig.~\ref{large_ex} with $P=10000$,
the ellipsoids of Type T are separated by at least $\delta = 0.2$, and driven by randomly sampled forces and torques \footnote{Since it is slow to generate such a large cluster of ellipsoids using true shortest distance computations, an alternative generation of the geometry is used. A cluster of unit spheres is first generated, with each sphere exactly $\delta$ from at least one other. Randomly oriented ellipsoids are then placed inside these spheres. Few ellipsoid pairs approach $\delta$ separation.}. The particles are discretized with $N_v = 34$.
The error is estimated as follows.
For the same type of cluster downsized to 5000 particles, the resulting rigid body velocities on the particles are determined with two grids: $N_v=34$ and $40$. 
The maximum relative velocity difference in the two computed rigid body motions is $\|\vec U_{N_v =34}-\vec U_{N_v = 40}\|_{\infty}/\|\vec U_{N_v = 40}\|_{\infty}= 3.99\times 10^{-6}$. Based on experience from other experiments, this error level is taken as a good estimate for the 10000-particle case.

\end{example}

\section{Conclusions}\label{conclusion}  

We present spectrally accurate well-conditioned solvers for the Laplace elastance and Stokes mobility problems involving a large number of smooth bodies, based on the method of fundamental solutions.
The formulations are free from additional constraints, achieved by projecting the linear space of source strengths into a) the subspace of constant vectors (for elastance) or the subspace of rigid body motions (for mobility), plus b) its orthogonal complement.
The former subspace directly controls the unknown constant potentials or rigid-body motions, while the latter generates a zero-net-charge (or zero-net-force-and-torque) solution potential to which a completion flow is added to account for the
known net charges (or forces and torques).
We call this a ``recompleted'' formulation.
With one-body preconditioning and FMM-acceleration, the scaling is linear in the number of particles, assuming identical particles, and problems with of order $10^4$ particles can be solved with of order one workstation-hour of computation.
For spheres, the number of MFS unknowns needed for a given accuracy is similar to that of the sphere-specialized boundary integral method of \cite{Corona2018,Yan2020}, while avoiding its elaborate analytic formulae.
Yet our method is more general, applying to any particle shape amenable to an MFS solution; for instance, many of our tests are for triaxial ellipsoids. Although not explicitly demonstrated through simulations in this paper, the method efficiently handles polydisperse systems as well.

 Our findings suggest several future directions. For random particle clusters of approximately constant density the GMRES iteration count stays constant with the number of particles for the mobility problem, rather than growing weakly as for the resistance problem \cite{Broms2024}, and we would like to understand this difference (see \cite{USABIAGA2016}).
  For ellipsoids, we have only tested particles with moderate aspect ratios, and it would be interesting to push this limit further. At high aspect ratio, a proportionally larger $N$ per body is expected, at which point it could be valuable to exploit azimuthal symmetry for axisymmetric particles in the one-body dense direct preconditioning method \cite{Liu2016}. Further, the present method could naturally be combined with the lubrication-adapted image systems of \cite{Broms2024}, enabling accurate mobility solutions at much closer separations $\delta$. The present formulation is also expected to aid in dynamic simulations of complex fluids and rheology. Efficient time-stepping will then need a contact-avoiding strategy \cite{Lu2019,Yan2020,Broms2024b}. It is also of great interest to investigate how to utilize the MFS to sample Brownian (thermal) hydrodynamic fluctuations.

\section*{Acknowledgments}
	 Broms and Tornberg acknowledge support from the Swedish Research Council: grant no.~2019-05206 and the research environment grant INTERFACE (biomaterials), no.~2016-06119.
We benefited from discussions with
Leslie Greengard, Dhairya Malhotra, David Krantz, Shravan Veerapaneni and Eduardo Corona, and from the input of the anonymous reviewers.
 The Flatiron Institute is a division of the Simons Foundation.  Broms is grateful for a research visit to CCM Flatiron where the initial ideas for this work were sparked.


 \appendix
  \section{MFS accuracy for spheres}
 
 \subsection{Comparison against a boundary integral scheme}\label{acc_spheres}
 
 For comparable numbers of degrees of freedom, we briefly compare the accuracies of our proposed MFS mobility solver and a sphere scheme based on spherical harmonics by duplicating tests in Yan et al.~\cite{Yan2020}.  In Figs.~\ref{grav_conv} and \ref{grav_conv2}, two unit spheres at $\vec c^{(1)} = [0,0,0]$ and $\vec c^{(2)} = [2+\delta,0,0]$ are affected by a gravitational force $\vec f^{(1)} = \vec f^{(2)} = [0,0,-F_g]$ and in Fig.~\ref{torque_rot}, by a torque $\vec t^{(1)} = \vec t^{(2)}  = [0, T, 0]$. In Figs.~\ref{grav_conv} and \ref{torque_rot}, the number of source points $N$ is varied, fixing proxy radius $R_p = 0.7$, while in Fig.~\ref{grav_conv2}, $N = 762$ is fixed and $R_p$ is varied. \cite[Fig.~2]{Yan2020} shows that at $\delta=0.2$, harmonic maximum degree $p=12$, hence $N=(p+1)(2p+1)=325$ surface unknowns, are needed to reach around $10^{-6}$ relative error. This is consistent with Fig.~\ref{grav_conv}, showing that at this same $\delta$, an error of $10^{-6}$ is reached with slightly more than $N=289$.
 Likewise, the numbers of unknowns for the torques in \cite[Fig.~3]{Yan2020} and our Fig.~\ref{torque_rot} are similar. At $\delta=0.05$, both methods give around $10^{-4}$ error using $N\approx 600$. In short, the MFS appears to have a similar accuracy for a given number of unknowns $N$ as a BIE using spherical harmonics.
 
 In these tests, our error level is computed relative to a fine reference computed with $N = 3600$, and is normalized by the single sphere 
 settling velocity for a single sphere as in \cite{Yan2020}.
 Unlike in that work, our error plateaus at a constant level dependent on $N$ at large $\delta$, attributed to aliasing due to the discreteness of the MFS sources. Such plateaus are never worse than $10^{-5}$ even for the smallest $N=201$. 
At smaller $\delta$, the error is instead dominated by the fact that the singularity in the analytic continuation of the Stokes solution at the surface in the interior domain is not enclosed by the MFS source points \cite{Barnett2007,Doicu2000}. These two error regimes were discussed in Remark \ref{r:rate}.

 \begin{figure}[h!]
 \centering
 		\begin{subfigure}[t]{0.33\textwidth}
			\centering
			\includegraphics[trim = {5.3cm 16.8cm 8.5cm 4.2cm},clip,width=1.1\textwidth]{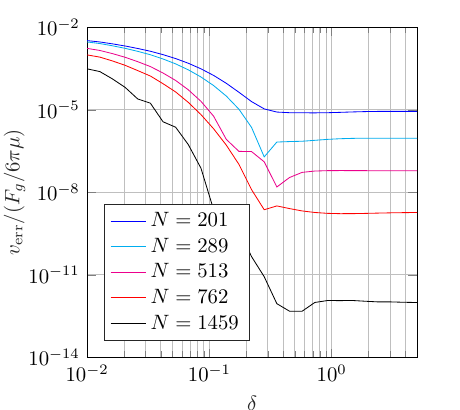}
			\caption{Error with a gravitational force.}
			\label{grav_conv}
		\end{subfigure}~~
   		\begin{subfigure}[t]{0.33\textwidth}
			\centering
		\includegraphics[trim = {5.3cm 16.8cm 8.5cm 4.2cm},clip,width=1.1\textwidth]{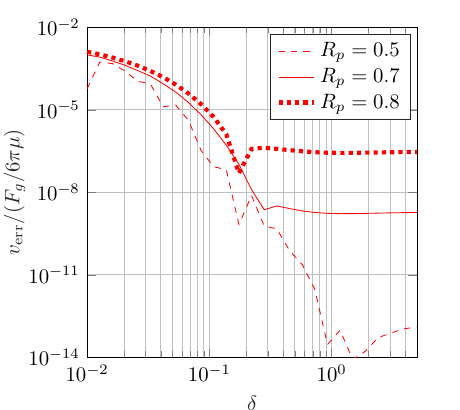}
			\caption{Error with a gravitational force.}
			\label{grav_conv2}
		\end{subfigure}~~
   		\begin{subfigure}[t]{0.33\textwidth}
			\centering
		\includegraphics[trim = {5.3cm 16.8cm 8.5cm 4.2cm},clip,width=1.1\textwidth]{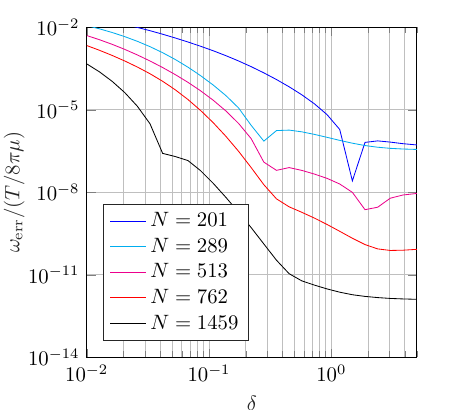}
			\caption{Error with an equal torque.}
			\label{torque_rot}
		\end{subfigure}
  \caption{Convergence study in the number of proxy points $N$ for two unit spheres separated by $\delta$. Errors are similar to those of the BIE spherical harmonics scheme of \cite{Yan2020} for comparable numbers of degrees of freedom. 
  In panels (a) and (c), the particles are affected by a gravitational force $F_g$ and the translational velocity error is taken, while in panel (b), the spheres are affected by an equal torque $T$ and the rotational velocity error is taken.}
  \label{conv_pair}
  \end{figure}

  \subsection{Comparison against the rigid multiblob method}\label{multiblob}

We repeat the experiment with five randomly positioned spheres presented in Fig.~\ref{Fig:sphere_QBXa}, employing the rigid multiblob method to compare its accuracy against that of our MFS scheme. The rigid multiblob method is an approximate solution technique widely used in the engineering community, thouroughly described in e.g.~\cite{USABIAGA2016}, where a collection of ``blobs'' approximates the particle surface. 
Mathematically, it leads to a saddle point system similar to that of the MFS, but with a regularized kernel---the Rotne–Prager–Yamakawa (RPY) tensor---replacing the Stokeslet. This similarity allows for the same solution strategy as with the MFS, where a recompleted formulation can be solved with one-body preconditioning (non-standard in the rigid multiblob literature). Unlike in MFS, the source and collocation points coincide, and must be placed slightly offset from the particle boundary to reproduce the correct far-field behavior; this also improves near-field accuracy \cite{Broms2022}. Following \cite{Broms2022}, we therefore optimize the discretization of the multiblob particles with respect to the offset and the regularization parameter. We then solve the mobility problem for the same 56 test geometries used in Section~\ref{BIEcomp}, which consist of five randomly positioned unit spheres driven by randomly sampled forces and torques. Accuracy is reported as a function of the minimum particle-particle distance $\delta$ in Figure~\ref{fig_multiblob}.
Despite optimization, and the inclusion of spherical design grids with the same number of degrees of freedom as the MFS, the multiblob method only achieved 1--4 digit accuracy, in contrast with the 4--8 digits using MFS.
  \begin{figure}[h!]
  		\centering
					\includegraphics[trim = {5cm 17cm 4.0cm 4.4cm},clip,width=0.66\textwidth]{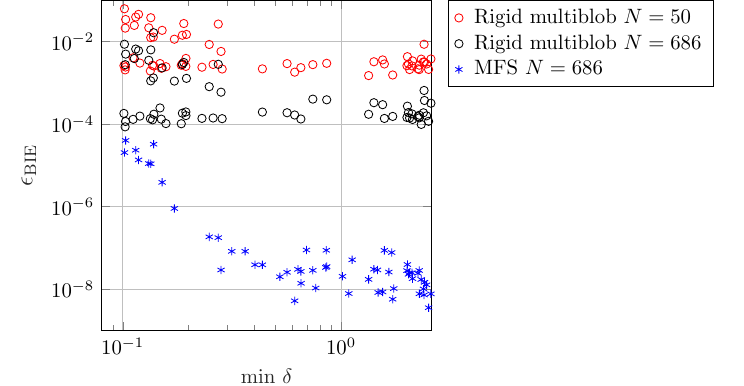}
			\caption{Comparison between our MFS scheme and the rigid multiblob method, each using $N$ sources per particle, for configurations of five randomly positioned spheres. Velocity errors $\epsilon_{\text{BIE}}$ as defined in~\eqref{BIEerr} are computed relative to the well-resolved BIE reference solution, and $\delta$ denotes the minimum particle-particle separation.}         
			\label{fig_multiblob}
            \end{figure}
	\bibliographystyle{myIEEEtran} 
	\bibliography{mobility_refs}

\end{document}